\documentclass[a4paper,11pt]{amsart}

\usepackage[T1]{fontenc}
\usepackage{amsmath, amssymb, amsthm}

\usepackage[top=3cm, bottom=3cm, left=2.5cm, right=2.5cm]{geometry}

\usepackage[linkbordercolor={0 0 1}, citebordercolor={0 1 0}]{hyperref}
\usepackage{color}

\newtheoremstyle{example}{\topsep}{\topsep}%
     {\slshape}
     {}
     {\bfseries}
     {.}
     {  }
     {}

\theoremstyle{plain}
\newtheorem{thm}{Theorem}[section]
\newtheorem*{thm*}{Theorem}
\newtheorem{lem}[thm]{Lemma}
\newtheorem{cor}[thm]{Corollary}
\newtheorem*{cor*}{Corollary}
\newtheorem{prop}[thm]{Proposition}
\newtheorem*{prop*}{Proposition}
\theoremstyle{definition}
\newtheorem{defn}[thm]{Definition}
\theoremstyle{remark}
\newtheorem{ex}[thm]{Example}
\newtheorem{rmk}[thm]{Remark}

\newcommand{\N}{\mathbb{N}}
\newcommand{\R}{\mathbb{R}}
\newcommand{\Q}{\mathbb{Q}}

\newcommand{\F}{\mathcal{F}}
\newcommand{\B}{\mathcal{B}}
\newcommand{\W}{\mathcal{W}}
\newcommand{\X}{\mathcal{X}}
\newcommand{\Y}{\mathcal{Y}}
\newcommand{\G}{\mathcal{G}}
\newcommand{\K}{\mathcal{K}}

\newcommand{\limM}{\lim_{M \rightarrow \infty}}

\title[Dominating Martingale Measures]{The Existence of Dominating Local Martingale Measures}

\author[P. Imkeller]{Peter Imkeller}
 
\author[N. Perkowski]{Nicolas Perkowski}
\address{Institut f\"{u}r Mathematik\\
Humboldt-Universit\"{a}t zu Berlin\\
Rudower Chaussee 25\\
12489 Berlin\\
Germany}
\email{imkeller@math.hu-berlin.de, perkowsk@math.hu-berlin.de}

\date{\today}

\begin{document}

\begin{abstract}
   We prove that, for locally bounded processes, absence of arbitrage opportunities of the first kind is equivalent to the existence of a dominating local martingale measure. This is related to and motivated by results from the theory of filtration enlargements.
\end{abstract}


\maketitle


\section{Introduction}

It may be argued that the foundation of financial mathematics consists in giving a mathematical characterization of market models satisfying certain financial axioms. This leads to so-called \emph{fundamental theorems of asset pricing}. Harrison and Pliska \cite{Harrison1981} were the first to observe that, on finite probability spaces, the absence of arbitrage opportunities (condition \emph{no arbitrage, (NA)}) is equivalent to the existence of an equivalent martingale measure. A definite version was shown by Delbaen and Schachermayer \cite{Delbaen1994}. Their result, commonly referred to as \emph{the} Fundamental Theorem of Asset Pricing, states that for locally bounded semimartingale models there exists an equivalent probability measure under which the price process is a local martingale, if and only if the market satisfies the condition \emph{no free lunch with vanishing risk (NFLVR)}. Delbaen and Schachermayer also observed that (NFLVR) is satisfied if and only if there are no arbitrage opportunities (i.e. (NA) holds), and if further it is not possible to make an unbounded profit with bounded risk (we say there are \emph{no arbitrage opportunities of the first kind}, condition \emph{(NA1)} holds). Since in finite discrete time, (NA) is equivalent to the existence of an equivalent martingale measure, it was then a natural question how to characterize continuous time market models satisfying only (NA) and not necessarily (NA1). For continuous price processes, this was achieved by Delbaen and Schachermayer \cite{Delbaen1995a}, who show that (NA) implies the existence of an \emph{absolutely continuous} local martingale measure.


Here we complement this program, by proving that for locally bounded processes, (NA1) is equivalent to the existence of a \emph{dominating} local martingale measure. Constructing dominating probability measures is rather delicate, and F\"ollmer's measure (\cite{Follmer1972}) associated to a nonnegative supermartingale appears naturally in this context.


Let us give a more precise description of the notions of arbitrage considered in this work, and of the obtained results.


Let $S$ be a $d$-dimensional stochastic process on a filtered probability space $(\Omega, \F, (\F_t)_{t \ge 0}, P)$. We assume throughout the paper that the filtration $(\F_t)$ is right-continuous, and that $\F = \F_\infty = \vee_{t \ge 0} \F_t$. We think of $S$ as the (discounted) price process of $d$ financial assets.

We should point out that the filtration $(\F_t)$ will \emph{not} be complete with respect to $P$. Our aim is to construct dominating measures which may charge $P$-null sets. Therefore we cannot complete the filtration by the $P$-null sets, hence we only assume that $(\F_t)$ is right-continuous. This means that we have to slightly deviate from the usual definition of a semimartingale. For us, a semimartingale is the sum of a local martingale and a process of finite variation on bounded intervals. However we only assume that semimartingales are almost surely (a.s.) c\`adl\`ag. A semimartingale does \emph{not} need to be c\`adl\`ag for \emph{every} $\omega \in \Omega$. Our definition follows Jacod and Shiryaev \cite{Jacod2003}, who also work with non-complete filtrations. We argue in Section \ref{sec: supermartingale densities} and Appendix \ref{app:complete filtration} that the non-completeness of our filtration will not pose any problem.


A \emph{strategy} is a predictable process $(H_t)_{t\ge 0}$ with values in $\R^d$. If $S$ is a semimartingale and $\lambda > 0$, then a strategy $H$ is called \emph{$\lambda$-admissible} (for $S$) if the stochastic integral $H\cdot S = \int_0^\cdot H_s \cdot dS_s$ exists and satisfies $P((H\cdot S)_t \ge -\lambda)=1$ for all $t \ge 0$. Here we write $a \cdot b = \sum_{i=1}^d a_i b_i$ for the usual inner product on $\R^d$.

Similarly, a \emph{simple strategy} is a process of the form $H_t = \sum_{j=0}^{n-1} F_j 1_{(T_j, T_{j+1}]}(t)$ for stopping times $0 \le T_0 < T_1 < \dots < T_n < \infty$ and bounded $\F_{T_j}$-measurable random variables $F_j$ with values in $\R^d$. If $S$ is a right-continuous adapted process, then the integral $H\cdot S$ is defined as
\begin{align*}
   (H\cdot S)_t = \sum_{j \ge 0} F_j (S_{T_{j+1}\wedge t} - S_{T_j\wedge t}),
\end{align*}
and $\lambda$-admissible strategies are defined analogously to the semimartingale case.

We denote by
\begin{align}\label{eq:W1}
   \W_1 = \{1 + (H\cdot S)_\cdot : H \text{ is a 1-admissible strategy and } (H\cdot S)_t \text{ a.s. converges as } t \rightarrow \infty \}
\end{align}
all wealth processes obtained by using 1-admissible strategies under initial wealth 1, and such that the terminal wealth is well defined. Similarly we define
\begin{align*}
   \W_{1,s} = \{1 + (H\cdot S)_\cdot : H \text{ is a 1-admissible simple strategy} \}.
\end{align*}
Note that the convergence condition in \eqref{eq:W1} is trivially satisfied for simple strategies. We will also need
\begin{align}\label{eq:K1}
   \K_1 = \{X_\infty : X \in \W_1 \} \qquad \text{and} \qquad \K_{1,s} = \{X_\infty : X \in \W_{1,s} \},
\end{align}
i.e. all terminal wealths that are attainable with initial wealth 1 and using 1-admissible strategies.

We write $L^0 = L^0(\Omega, \F, P)$ for the space of real-valued random variables on $(\Omega, \F)$, where we identify random variables that are $P$-almost surely equal. We equip $L^0$ with the distance $d(X,Y) = E(|X-Y|\wedge 1)$, under which it becomes a complete metric space.

Recall that a family of random variables $\X$ is called \emph{bounded in probability}, or \emph{bounded in $L^0$}, if
\begin{align*}
   \limM \sup_{X \in \X} P(|X| \ge M) = 0.
\end{align*}

\begin{defn}
   We say that a semimartingale $S$ satisfies \emph{no arbitrage of the first kind (NA1)} if $\K_1$ is bounded in probability. We say that $S$ satisfies \emph{no arbitrage (NA)} if there is no $X \in \K_1$ with $X \ge 1$ and $P(X > 1) > 0$, . If both (NA1) and (NA) hold, we say that $S$ satisfies \emph{no free lunch with vanishing risk (NFLVR)}.
   
   Similarly we say that a right-continuous adapted process $S$ satisfies \emph{no arbitrage of the first kind with simple strategies (NA1$_s$)}, \emph{no arbitrage with simple strategies (NA$_s$)}, or \emph{no free lunch with vanishing risk with simple strategies (NFLVR$_s$)}, if $\K_{1,s}$ satisfies the corresponding conditions.
\end{defn}

Heuristically, (NA) says that it is not possible to make a profit without taking a risk. (NA1) says that is not possible to make an unbounded profit if the risk remains bounded. This is why (NA1) is also referred to as ``no unbounded profit with bounded risk'' (NUPBR), see for example Karatzas and Kardaras \cite{Karatzas2007}.

The main result of this paper is that for locally bounded semimartingales $S$, (NA1) is equivalent to the existence of a dominating local martingale measure. As a byproduct of the proof, we obtain that a locally bounded, right-continuous, and adapted process $S$ that satisfies (NA1$_s$) is already a semimartingale, and in this case $S$ also satisfies (NA1).

When constructing absolutely continuous probability measures, it suffices to work with random variables. In Section \ref{sec: motivation} below, we argue that dominating measures correspond to nonnegative supermartingales with strictly positive terminal values. We also show that a dominating local martingale measure corresponds to a supermartingale density in the following sense.

\begin{defn}\label{def:supermartingale density}
   Let $\Y$ be a family of stochastic processes. A \emph{supermartingale density} for $\Y$ is an almost surely c\`adl\`ag and nonnegative supermartingale $Z$ with $Z_\infty = \lim_{t \rightarrow \infty} Z_t > 0$ a.s., such that $YZ$ is a supermartingale for every $Y\in \Y$.
   
   If all processes in $\Y$ are of the form $1 + (H\cdot S)$ for suitable integrands $H$, and if $Z$ is a supermartingale density for $\Y$, then we will sometimes call $Z$ a supermartingale density for $S$.
\end{defn}

In the literature, supermartingale densities are usually referred to as \emph{supermartingale deflators}. We think of a supermartingale density as the ``Radon-Nikodym derivative'' $dQ/dP$ of a dominating measure $Q \gg P$. This is why we prefer the term supermartingale density.

First we give an alternative proof of a well-known result.

\begin{thm}\label{thm: ex supermartingale density}
   Let $S$ be a $d$-dimensional adapted process, a.s. right-continuous (resp. a $d$-dimensional semimartingale). Then (NA1$_s$) (resp. (NA1)) holds if and only if there exists a supermartingale density for $\W_{1,s}$ (resp. for $\W_1$).
\end{thm}

As a consequence, (NA1$_s$) implies the semimartingale property for locally bounded processes.

\begin{cor}\label{cor:NA1 implies semimartingale}
   Let $S$ be a $d$-dimensional adapted process, a.s. right-continuous. If every component $S^i$ of $S=(S^1,\dots, S^d)$ is locally bounded from below and if $S$ satisfies (NA1$_s$), then $S$ is a semimartingale that satisfies (NA1), and any supermartingale density for $\W_{1,s}$ is also a supermartingale density for $\W_{1}$.
\end{cor}

%

Given a supermartingale density $Z$ for $S$, we then apply Yoeurp's \cite{Yoeurp1985} results on F\"ollmer's measure \cite{Follmer1972}, to construct a dominating measure $Q \gg P$ associated to $Z$. We define $\gamma$ to be a right-continuous version of the density process $\gamma_t = dP/dQ|_{\F_t}$, and $T$ to be the first time that $\gamma$ hits zero, $T = \inf \{t \ge 0: \gamma_t = 0\}$. Set
\begin{align*}
   S^{T-}_t = S_t 1_{\{t<T\}} + S_{T-} 1_{\{t \ge T\}} = S_t 1_{\{t<T\}} + \lim_{s \rightarrow T-} S_s 1_{\{t \ge T\}}.
\end{align*}
Note that $S$ and $S^{T-}$ are $P$-indistinguishable.

In the predictable case we then obtain the following result.

\begin{thm}\label{thm:predictable supermartingale density}
    Let $S$ be a predictable semimartingale. If $Z$ is a supermartingale density for $\W_1$, then $Z$ determines a probability measure $Q \gg P$ such that $S^{T-}$ is a $Q$-local martingale. Conversely, if $Q \gg P$ is a dominating local martingale measure for $S^{T-}$, then $\W_1$ admits a supermartingale density.
\end{thm}

Theorem \ref{thm:predictable supermartingale density} is false if $S$ is not predictable, as we will demonstrate on a simple example. But in the non-predictable case we are able to exhibit a subclass of supermartingale densities that \emph{do} give rise to dominating local martingale measures. Conversely every dominating local martingale measure corresponds to a supermartingale density, even for processes that are not predictable. Therefore the following theorem, the main result of this paper, is valid for all locally bounded processes that are adapted and a.s. right-continuous. In the non-predictable case we build on results of \cite{Takaoka2012} that are only formulated for processes on finite time intervals. So in the theorem we let $T_\infty = \infty$ if $S$ is predictable, and $T_\infty \in (0,\infty)$ otherwise.

\begin{thm}\label{thm:main result}
   Let $(S_t)_{t \in [0,T_\infty]}$  be a locally bounded, adapted process, a.s. right-continuous. Then $S$ satisfies (NA1$_s$) if and only if there exists a dominating $Q\gg P$, such that $S^{T-}$ is a $Q$-local martingale.
\end{thm}

This work is motivated by insights from the theory of filtrations enlargements. A filtration $(\G_t)$ is called \emph{filtration enlargement} of $(\F_t)$ if $\G_t \supseteq \F_t$ for all $t \ge 0$. A basic question is then under which conditions all members of a given family of $(\F_t)$-semimartingales are $(\G_t)$-semimartingales. We say that \emph{Hypoth\`{e}se $(H')$} is satisfied if \emph{all} $(\F_t)$-semimartingales are $(\G_t)$-semimartingales. Given a $(\F_t)$-semimartingale that satisfies (NFLVR), i.e. for which there exists an equivalent local martingale measure, one might also ask under which conditions it still satisfies (NFLVR) under~$(\G_t)$. It is well known, and we llustrate this in an example below, that the (NFLVR) condition is usually violated after filtration enlargements.

However it turns out that (NA1) is relatively stable under filtration enlargements. If $(\G_t)$ is an initial enlargement of $(\F_t)$, i.e. $\G_t = \F_t \vee \sigma(L)$ for some random variable $L$, then Jacod's criterion \cite{Jacod1985} is a celebrated condition that guarantees Hypoth\'ese $(H')$ to hold. We show that Jacod's criterion implies in fact the existence of a \emph{universal supermartingale density}. A strictly positive process $Z$ is called universal supermartingale density if $Z M$ is a $(\G_t)$-supermartingale for every nonnegative $(\F_t)$-supermartingale $M$. This is of course a much stronger than Hypoth\`{e}se $(H')$, and in particular it implies that every process satisfying (NA1) under $(\F_t)$ also satisfies (NA1) under $(\G_t)$.

We also show that if $(\G_t)$ is a general (not necessarily initial) filtration enlargement of $(\F_t)$, and if there exists a universal supermartingale density for $(\G_t)$, then a generalized version of Jacod's criterion is necessarily satisfied.

\subsection{Structure of the paper}

Section \ref{sec: motivation} describes our motivation from filtration enlargements in more detail. In Section \ref{sec: motivation} we also argue that dominating local martingale measures should correspond to supermartingale densities. In Section \ref{sec: supermartingale densities} we prove that the existence of supermartingale densities is equivalent to (NA1$_s$). In Section \ref{sec: dominating measures} we prove if $S$ is predictable, then $Z$ is a supermartingale density for $S$ if and only if $S^{T-}$ is a local martingale under the F\"ollmer measure $P^Z$. We also prove our main result, Theorem \ref{thm:main result}, for general locally bounded processes (not necessarily predictable). In Section \ref{sec: jacod criterion} we return to filtration enlargements and examine how Jacod's criterion relates to our results.

\subsection{Relevant literature}

Supermartingale densities were first considered by Kramkov and Schachermayer \cite{Kramkov1999} and Becherer \cite{Becherer2001}.

The semimartingale case of Theorem \ref{thm: ex supermartingale density} was shown by Karatzas and Kardaras \cite{Karatzas2007}. Their proof extensively uses the semimartingale characteristics of $S$, and can therefore not be applied to general processes satisfying (NA1$_s$). Note that Corollary \ref{cor:NA1 implies semimartingale} states that any locally bounded process satisfying (NA1$_s$) is a semimartingale. But for unbounded processes this is no longer true, as we shall demonstrate in a simple counterexample below. A more general result than Theorem \ref{thm: ex supermartingale density} is shown in Rokhlin \cite{Rokhlin2010}, using arguments that are related to our proof. In fact our arguments are powerful enough to imply the results of \cite{Rokhlin2010}. We were not aware of either of these works before completing our proof, and decided to keep it in the paper because we believe that it gives a nice application of \emph{convex compactness}, as introduced by Zitkovic \cite{Zitkovic2010}. Oversimplifying things a bit, one can understand convex compactness as an elegant way of formalizing convergence and compactness results that are usually shown by ad-hoc considerations based on results like Lemma A1.1 of \cite{Delbaen1994}. We also believe that our techniques may be interesting in more complicated contexts, say under transaction costs, where arbitrage considerations no longer imply the the semimartingale property of the price process. 

It is well known that a locally bounded process satisfying (NA1$_s$) must be a semimartingale, see Ankirchner's thesis \cite{Ankirchner2005}, Theorem 7.4.3, and also Kardaras and Platen \cite{Kardaras2011a}. See also \cite{Delbaen1994} for a first result in this direction. This part of Corollary \ref{cor:NA1 implies semimartingale} is an immediate consequence of Theorem \ref{thm: ex supermartingale density}. We rely on \cite{Kardaras2011a} to obtain that (NA1$_s$) implies (NA1) for locally bounded processes, and that in that case supermartingale densities for $\W_{1,s}$ are supermartingale densities for $\W_1$.

Recently there has been an increased interest in F\"ollmer's measure, motivated by problems from mathematical finance. F\"ollmer's measure appears naturally in the construction and study of \emph{strict local martingales}, i.e. local martingales that are not martingales. These are used to model bubbles in financial markets, see Jarrow, Protter and Shimbo \cite{Jarrow2010}. A pioneering work on the relation between F\"ollmer's measure and strict local martingales is Delbaen and Schachermayer \cite{Delbaen1995}. Other references are Pal and Protter \cite{Pal2010} and Kardaras, Kreher and Nikeghbali \cite{Kardaras2011}. The work most related to ours is Ruf \cite{Ruf2010}, where it is shown that in a diffusion setting, (NA1) implies the existence of a dominating local martingale measure. All these works have in common that they study F\"ollmer measures of strictly positive local martingales. Carr, Fisher and Ruf \cite{Carr2011} study the F\"ollmer measure of a local martingale which is not strictly positive.

To the best of our knowledge, the current work is the first time that the F\"ollmer measure of an actual supermartingale (i.e. a supermartingale which is not a local martingale) is used as a local martingale measure. In F\"ollmer and Gundel \cite{Follmer2006}, supermartingales $Z$ are associated to ``extended martingale measures'' $P^Z$. But they define $P^Z$ to be an extended martingale measure if and only if $Z$ is a supermartingale density. This does not obiously imply that $S^{T-}$ or $S$ is a local martingale under $P^Z$ (and in general this is not true). Here we show that if $S$ is predictable, then any supermartingale density $Z$ corresponds to a dominating local martingale measure $P^Z$ - meaning that $S^{T-}$ is a local martingale under $P^Z$. For non-predictable $S$ we give a counterexample. In that case we identify a subclass of supermartingale densities that correspond to local martingale measures.

Another related work is Kardaras \cite{Kardaras2010}, where it is shown that (NA1) is equivalent to the existence of a finitely additive equivalent local martingale measure. Here we construct countably additive measures, that are however not equivalent but only dominating.

The main motivation for this work comes from the theory of filtrations enlargements, see for example Amendinger, Imkeller and Schweizer \cite{Amendinger1998}, Ankirchner's thesis \cite{Ankirchner2005}, and Ankirchner, Dereich and Imkeller \cite{Ankirchner2006}. In these works it is shown that if $M$ is a continuous local martingale in a given filtration $(\F_t)$, then under an enlarged filtration $(\G_t)$, assuming suitable conditions, $M$ is of the form $M = \tilde{M} + \int_0^\cdot \alpha_s d \langle \tilde{M} \rangle_s$, where $\tilde{M}$ is a $(\G_t)$-local martingale. It is then a natural question to ask whether there exists an equivalent measure $Q$ that ``eliminates'' the drift, i.e. under which $M$ is a $(\G_t)$-local martingale. In general the answer to this question is negative. However Ankirchner \cite{Ankirchner2005}, Theorem 9.2.7, observed that if there exists a well-posed utility maximization problem in the large filtration, then the \emph{information drift} $\alpha$ must be locally square integrable with respect to $\tilde{M}$. Here we show that for continuous processes, the square integrability of the information drift is \emph{equivalent} to the well-posedness of a utility maximization problem in the large filtration, we relate these conditions to (NA1), and we show that this allows to construct dominating local martingale measures. We also give the corresponding results for discontinuous processes.

\section{Motivation}
\label{sec: motivation}

In this section we show that under filtration enlargements, generally there exists no longer an equivalent local martingale measure. 
Then we recall that as long as Jacod's condition holds, there still is a dominating local martingale measure. Finally we argue that under Jacod's condition, (NA1) is often satisfied in the large filtration. We hope that this convinces the reader that (NA1) resp. (NA1$_s$) should be in some relation to the existence of dominating local martingale measures. Assuming that a dominating local martingale measure exists, we examine its Kunita-Yoeurp decomposition under $P$, and we see that it corresponds to a supermartingale density.

\subsection*{Equivalent local martingale measures and filtration enlargements}

Consider a filtered probability space $(\Omega, \F, (\F_t)_{t\ge 0}, P)$ with $P(A) \in \{0,1\}$ for every $A \in \F_0$. Define $\F_\infty = \vee_{t\ge 0} \F_t$. Let $S$ be a one-dimensional semimartingale that describes a complete market (i.e. for every $F \in L^\infty(\F_\infty)$ there exists a predictable process $H$, integrable with respect to $S$, such that $F = F_0 + \int_0^\infty H_s dS_s$ for some constant $F_0 \in \R$). Let $L$ be a random variable that is $\F_\infty$-measurable. Assume that $L$ is not $P$-a.s. constant. Define the initially enlarged filtration
\begin{align*}
   (\G_t = \F_t \vee \sigma (L): t \ge 0).
\end{align*}
This is a toy model for insider trading. At time 0, the insider has the additional knowledge of the value of $L$. Since $L$ is not constant, there exists $A \in \sigma(L)$ such that $P(A) \in (0,1)$. Assume $Q$ is an equivalent $(\G_t)$-local martingale measure for $S$. Consider the $(Q,(\F_t))$-martingale $N_t = E_Q(1_A|\F_t)$, $t \ge 0$. Since the market is complete, $1_A$ can be replicated. That is, there exists a $(\F_t)$-predictable strategy $H$ such that $N_\cdot = Q(A) + \int_0^\cdot H_s dS_s$. But then $\int_0^\cdot H_s dS_s$ is a bounded $(Q, (\G_t))$-local martingale. Hence it is a martingale, and since $A^c \in \G_0$, we obtain
\begin{align*}
   0 = E_Q(1_{A^c} 1_A) = E_Q\left(1_{A^c} \left( Q(A) + \int_0^\infty H_s dS_s\right)\right) = Q(A^c) Q(A) > 0,
\end{align*}
which is absurd. The last step follows because $Q$ was assumed to be equivalent to $P$.

So already in the simplest insider trading models there may not exist an equivalent local martingale measure. If $S$ is locally bounded, then by the Fundamental Theorem of Asset Pricing at least one of the conditions (NA) or (NA1) has to be violated.

\subsection*{Jacod's criterion and dominating local martingale measures}

Let $(\G_t)$ be a filtration enlargement of $(\F_t)$, i.e. $\F_t \subseteq \G_t$ for every $t \ge 0$. Let $\mathcal{S}$ be a family of $(\F_t)$-semimartingales. One of the typical questions in filtration enlargements is under which conditions all $S \in \mathcal{S}$ are $(\G_t)$-semimartingales. \emph{Hypoth\`{e}se $(H')$} is said to be satisfied if \emph{all} $(\F_t)$-semimartingales are $(\G_t)$-semimartingales. 

Jacod's criterion \cite{Jacod1985} is a famous condition that implies Hypoth\`{e}se $(H')$. Here we give an equivalent formulation, first found by F\"ollmer and Imkeller \cite{Follmer1993} and later generalized and carefully studied by Ankirchner, Dereich and Imkeller \cite{Ankirchner2007}. Let $L$ be a random variable and consider the initial enlargement $\G_t = \F_t \vee \sigma(L)$. Define the product space
\begin{align*}
   \overline{\Omega} = \Omega \times \Omega, \qquad \overline{\G} = \F_\infty \otimes \sigma(L), \qquad \overline{\G}_t = \F_t \otimes \sigma(L).
\end{align*}
We define two measures on $\overline{\Omega}$. The decoupling measure $\overline{Q} = P|_{\F_\infty} \otimes P|_{\sigma(L)}$, and $\overline{P} = P \circ \psi^{-1}$, where $\psi: \Omega \rightarrow \overline{\Omega}$, $\psi(\omega) = (\omega, \omega)$.  We then have the following result, which in this setting is just a reformulation of Jacod's criterion.

\begin{thm*}[Theorem 1 in \cite{Ankirchner2007}]
   If $\overline{P} \ll \overline{Q}$, then Hypoth\`ese $(H')$ holds, i.e. any $(\F_t)$-semimartingale is a $(\G_t)$-semimartingale.
\end{thm*}

In this formulation it is quite obvious why Jacod's criterion works. Under the measure $\overline{Q}$, the additional information from $L$ is independent of $\F_\infty$. Therefore any $(\F_t)$-martingale $M$ will stay a $(\overline{\G}_t)$-martingale under $\overline{Q}$ (if we embed $M$ from $\Omega$ to $\overline{\Omega}$ by setting $\overline{M}_t(\omega, \omega') = M_t(\omega)$). By assumption, $\overline{Q}$ \emph{dominates} $\overline{P}$. So an application of Girsanov's theorem implies that $\overline{M}$ is a $\overline{P}$-semimartingale. But it is possible to show that any $(\overline{P}, (\overline{\G}_t))$-semimartingale is a $(P, (\G_t))$-semimartingale, which completes the argument. The message is that Jacod's criterion implies the existence of a dominating measure under which any $(\F_t)$-martingale is a $(\G_t)$-martingale.

It is not hard to see that Jacod's criterion is always satisfied if $L$ takes its values in a countable set, regardless of the structure of $(\Omega, \F, (\F_t), P)$ and $S$. So if we recall our example of an initial filtration enlargement in a complete market from above, then we observe that Jacod's criterion may be satisfied even though there is no equivalent local martingale measure in the large filtration.

\subsection*{Utility maximization and filtration enlargements}

There are many articles devoted to calculating the additional utility of an insider. Assume $S$ is a semimartingale in the large filtration $(\G_t)$. Then we define the set of attainable terminal wealths $\K_1(\F_t)$ and $\K_1(\G_t)$ as in \eqref{eq:K1}, using $(\F_t)$-predictable and $(\G_t)$-predictable strategies respectively.

If $S$ describes a complete market under $(\F_t)$, and if $(\G_t)$ is an initial enlargement satisfying Jacod's criterion, then it is shown in Ankirchner's thesis (\cite{Ankirchner2005}, Theorem 12.6.1, see also \cite{Ankirchner2006}), that the maximal expected logarithmic utility under $(\G_t)$ is given by
\begin{align*}
   \sup_{X \in \K_1(\G_t)} E(\log(X)) = \sup_{X \in \K_1(\F_t)} E(\log(X)) + I(L, \F_\infty),
\end{align*}
where $I(L, \F_\infty)$ is the mutual information between $L$ and $\F_\infty$. This mutual information may be finite, and therefore the maximal expected utility under $(\G_t)$ may be finite. But we show in Proposition \ref{prop: l0-bded iff finite utility} below that finite utility and (NA1) are equivalent.

\begin{lem}\label{lem:utility concrete}
   $S$ satisfies (NA1) under $(\G_t)$ if and only if there exists an unbounded increasing function $U$ such that the maximal expected utility is finite, i.e. such that
   \begin{align*}
      \sup_{X \in \K_1(\G_t)} E(U(X)) < \infty.
   \end{align*}
\end{lem}

\begin{proof}
   This is an immediate consequence of Proposition \ref{prop: l0-bded iff finite utility} below.
\end{proof}

In conclusion we showed that the (NFLVR) condition is not very robust with respect to filtration enlargements. Since (NFLVR) is equivalent to (NA) and (NA1), either (NA) or (NA1) must be violated after a typical filtration enlargement. We observed that the maximal expected logarithmic utility under an enlarged filtration may well be finite, and that this implies (NA1). Therefore we conclude that (NA) is the part of (NFLVR) that is less robust with respect to filtration enlargements (see Remark \ref{rmk:filtration discussion} for a more thorough discussion). Moreover in the examples where (NA1) holds, Jacod's criterion is satisfied as well. As we saw above, Jacod's criterion implies the existence of a dominating local martingale measure. Hence (NA1) seems to be related to the existence of a dominating local martingale measure. In this paper we prove that the two conditions are equivalent.

\subsection*{Supermartingale densities}

Now let us assume that we are given a dominating local martingale measure $Q \gg P$, and let us examine what type of object this gives us under $P$. We consider a fixed right-continuous filtration $(\F_t)$, and we assume that $S$ is a local martingale under $Q$ with $P \ll Q$. Define $\gamma$ to be the right-continuous density process, $\gamma_t = dP/dQ|_{\F_t}$. Then $T = \inf\{t \ge 0: \gamma_t = 0\}$ is a stopping time, and we can define the adapted process $Z_t = 1_{\{ t < T\}}/\gamma_t$. Let $H$ be 1-admissible for $S$ under $Q$, that is $Q(\int_0^{t} H_s dS_s \ge -1)=1$ for all $t \ge 0$. Let $s,t\ge 0$ and let $A \in \F_t$. We have
\begin{align}\label{eq:supermartingale density derivation}
   E_P(1_{A} Z_{t+s}(1 + (H\cdot S)_{t+s})) & = E_Q\left( \gamma_{t+s} 1_{A} \frac{1_{\{t+s < T\}}}{\gamma_{t+s}}(1 + (H\cdot S)_{t+s})\right) \\ \nonumber
   & \le E_Q\left( 1_{A} 1_{\{t < T\}}(1 + (H\cdot S)_{t+s}) \right) \\ \nonumber
   & \le E_Q\left( 1_{A} 1_{\{t < T\}}(1 + (H\cdot S)_{t}) \right) \\ \nonumber
   & = E_P(1_{A} Z_t (1 + (H\cdot S)_t))
\end{align}
using in the second line that $1_{A}(1+ (H\cdot S)_{t+s})$ is nonnegative, and in the third line that $1 + (H\cdot S)$ is a nonnegative $Q$-local martingale and therefore a $Q$-supermartingale. This indicates that $Z$ should be a supermartingale density. Of course here we only considered strategies that are 1-admissible under $Q$, and there might be strategies that are 1-admissible under $P$ but not under $Q$. The way to deal with this problem is to consider $S^{T-}$ rather than $S$. We will make this rigorous later.

Note that the couple $(Z,T)$ is the \emph{Kunita-Yoeurp decomposition} of $Q$ with respect to $P$. The Kunita-Yoeurp decomposition is a progressive Lebesgue decomposition on filtered probability spaces. It was introduced in Kunita \cite{Kunita1976} in a Markovian context, and generalized to arbitrary filtered probability spaces in Yoeurp \cite{Yoeurp1985}. Namely we have for every $t \ge 0$
\begin{enumerate}
   \item $P(T = \infty) = 1$,
   \item $Q(\cdot \cap \{T \le t\})$ and $P$ are mutually singular on $\F_t$,
   \item $Q(\cdot \cap \{T > t\})|_{\F_t} \ll P|_{\F_t}$, and for $A \in \F_t$ we have $Q(A \cap \{T>t\}) = E_P(1_A Z_t)$.
\end{enumerate}

Hence our program will be to find a supermartingale density $Z$, and to construct a measure $Q$ and a stopping time $T$, such that $(Z,T)$ is the Kunita-Yoeurp decomposition of $Q$ with respect to $P$. But the second part was already solved by \cite{Yoeurp1985}, and $Q$ will be the F\"ollmer measure of $Z$. After studying the relation between $S$ and $Z$, we will see that $S^{T-}$ is a local martingale under $Q$.

Before doing so, let us prove Lemma \ref{lem:utility concrete}. This is an immediate consequence of the following de la Vall\'{e}e-Poussin type theorem for families of random variables that are bounded in $L^0$.
\begin{prop} \label{prop: l0-bded iff finite utility}
   A family of random variables $\X$ is bounded in probability if and only if there exists a nondecreasing and unbounded function $U$ on $[0, \infty)$, such that
   \begin{align*}
      \sup_{X \in \X} E(U(|X|)) < \infty.
   \end{align*}
   In this case $U$ can be chosen concave and such that $U(0) = 0$.
\end{prop}
\begin{proof}
   First, assume that such a $U$ exists. Then
   \begin{align*}
      \sup_{X \in \X} P(|X| \ge M) \le \sup_{X \in \X} P(U(|X|) \ge U(M)) \le \frac{\sup_{X \in \X} E(U(|X|))}{U(M)}.
   \end{align*}
   Since $U$ is unbounded, the right hand side converges to zero as $M$ tends to $\infty$.
   
   Conversely, assume that $\X$ is bounded in probability. We need to construct a nondecreasing, unbounded, and concave function $U$ with $U(0)=0$, such that $E(U(|X|))$ is bounded for $X$ running through $\X$. Our construction is inspired by the proof of de la Vall\'{e}-Poussin's theorem. That is, we will construct a function $U$ of the form
   \begin{align*}
      U(x) = \int_0^x g(y) dy \qquad \text{where} \qquad g(y) = g_n, y \in [n-1, n)
   \end{align*}
   for a decreasing sequence of positive numbers $g_n$. This $U$ will be increasing, concave, $U(0) = 0$. It will be unbounded if and only if $\sum_{n=1}^\infty g_n = \infty$.

   For $U$ of this form we have by monotone convergence and Fubini (all terms are nonnegative)
   \begin{align*}
      E(U(|X|)) & = \sum_{n=1}^\infty E(U(|X|) 1_{\{ |X| \in [n-1, n)\}} ) \le \sum_{n=1}^\infty U(n) P( |X| \in [n-1, n) ) \\
      & = \sum_{n=1}^\infty \sum_{k=1}^n g_k P( |X| \in [n-1, n) ) = \sum_{k=1}^\infty \sum_{n=k}^\infty g_k P(|X| \in [n-1,n)) \\
      & = \sum_{k=1}^\infty g_k P(|X| \ge k-1) \le \sum_{k=1}^\infty g_k F_\X(k-1),
   \end{align*}
   where $F_\X(k-1) = \sup_{X \in \X} P(|X| \ge k-1)$.
   
   So the proof is complete if we can find a decreasing sequence $(g_k)$ of positive numbers, such that $\sum_{k=1}^\infty g_k = \infty$ but $\sum_{k=1}^\infty g_k F_\X(k-1) < \infty$. Let $n \in \N$. By assumption $(F_\X(k))$ converges to zero as $k \rightarrow \infty$, and therefore it also converges to zero in the Ces\`{a}ro sense. So we obtain for large enough $K_n$
   \begin{align}\label{eq:cesaro-conv}
      \frac{1}{K_n}\sum_{k=1}^{K_n} F_\X(k-1) \le \frac{1}{n}.
   \end{align}
   We choose an increasing sequence of numbers $K_n \ge n$, such that every $K_n$ satisfies \eqref{eq:cesaro-conv}. Define
   \begin{align*}
      g^n_k  = \begin{cases}
                          \frac{1}{nK_n}, &  k \le K_n \\
                          0, & k > K_n, 
                       \end{cases}
   \end{align*}
   and let $n_k$ denotes the smallest $n$ for which $g^n_k \neq 0$, i.e. the smallest $n$ for which $K_n \ge k$. The sequence $(K_n)$ is increasing, and therefore $n_k \le n_{k+1}$ for all $k$. Then the sequence $(g_k)$, where
   \begin{align*}
      g_k = \sum_{n=1}^\infty g^n_k = \sum_{n = n_k}^\infty \frac{1}{n K_n} \le \sum_{n=n_k}^\infty \frac{1}{n^2} < \infty,
   \end{align*}
   is decreasing in $k$. Moreover we have by Fubini
   \begin{align*}
      \sum_{k=1}^\infty g_k = \sum_{k=1}^\infty \sum_{n=1}^\infty g^n_k = \sum_{n=1}^\infty \sum_{k=1}^\infty g^n_k = \sum_{n=1}^\infty \sum_{k=1}^{K_n} \frac{1}{nK_n} = \sum_{n=1}^\infty \frac{1}{n} = \infty,
   \end{align*}
   and at the same time we get from \eqref{eq:cesaro-conv}
   \begin{align*}
      \sum_{k=1}^\infty g_k F_\X(k-1) = \sum_{n=1}^\infty \sum_{k=1}^{K_n} \frac{F_\X(k-1)}{nK_n} \le \sum_{n=1}^\infty \frac{1}{n^2} < \infty,
   \end{align*}
   which completes the proof.
\end{proof}

\begin{rmk}
   In Loewenstein and Willard \cite{Loewenstein2000}, Theorem 1, it is shown that the utility maximization problem for It\^{o} processes is well posed if and only if there is absence of a certain notion of arbitrage. They describe the critical arbitrage opportunities very precisely, and they consider more general utility maximization problems, allowing for intermediate consumption. Proposition~\ref{prop: l0-bded iff finite utility} is much simpler and more obvious, but therefore also more robust. It is applicable in virtually any context, say to discontinuous price processes that are not semimartingales, with transaction costs, and under trading constraints. The family of portfolios need not even be convex.
\end{rmk}

\begin{rmk}
   Note that supermartingale densities are the dual variables in the utility maximization problem, see \cite{Kramkov1999}. So taking Proposition \ref{prop: l0-bded iff finite utility} into account, Theorem \ref{thm: ex supermartingale density} states that the utility maximization problem is non degenerate if and only if the space of dual minimizers is nonempty. This insight might also be useful in more complicated contexts, say in market with transaction costs. As a sort of meta-theorem holding for many utility maximization problems, we expect that the space of dual variables is nonempty if and only if the space of primal variables is bounded in probability.
\end{rmk}

A first corollary is that any locally bounded process satisfying (NA1$_s$) is a semimartingale.

\begin{cor}
   Let $S$ be a locally bounded, c\`adl\`ag process satisfying (NA1$_s$). Then $S$ is a semimartingale.
\end{cor}
\begin{proof}
   Since $\K_{1,s}$ is bounded in probability, Proposition \ref{prop: l0-bded iff finite utility} implies that there exists an unbounded utility function $U$ for which $\sup_{X \in \K_{1,s}} E(U(X))<\infty$. Theorem 7.4.3 of \cite{Ankirchner2005} then implies that $S$ is a semimartingale.
\end{proof}

This result will also be an immediate consequence of Theorem \ref{thm: ex supermartingale density}.

\section{Existence of supermartingale densities}
\label{sec: supermartingale densities}

Now let us prove Theorem \ref{thm: ex supermartingale density} Let $(\Omega, \F, (\F_t)_{t \ge 0}, P)$ be a filtered probability space with a right-continuous filtration. We do not require $(\F_t)$ to be complete. This goes against a long tradition in probability theory to only work with filtrations satisfying the usual conditions. The most important reasons to consider complete filtrations are that the cross-section theorem (\cite{Dellacherie1980}, 44) only holds in complete $\sigma$-algebras, and as a consequence entrance times into Borel sets are generally only stopping times with respect to complete filtrations, and that supermartingales only have c\`adl\`ag modifications in complete filtrations.

However there are at least two classical books in stochastic analysis that avoid using complete filtrations as much as possible, Jacod \cite{Jacod1979} and Jacod and Shiryaev \cite{Jacod2003}. With non-complete filtrations one can obtain results that are nearly as powerful as the ones for complete filtrations. For example every stopping time $T$ in the completed filtration $(\F^P_t)$ is $P$-a.s. equal to a $(\F_t)$-stopping time $\tilde{T}$. And it is easy to see that entrance times of right-continuous processes into open or closed sets are hitting times as long as $(\F_t)$ is right-continuous. If $(\F_t)$ is right-continuous, then any supermartingale $Z$ with right-continuous expectation $t \mapsto E(Z_t)$ has a modification that is right-continuous for \emph{every} $\omega \in \Omega$, and which $P$-a.s. has left limits, see Remark I.1.37 of \cite{Jacod2003}. Note also that in \cite{Jacod2003} stochastic integration is done for non-complete filtrations. In Appendix \ref{app:complete filtration} we moreover recall that for every $(\F^P_t)$-adapted process that is a.s. c\`adl\`ag there exists an indistinguishable $(\F_t)$-adapted process, and similar results hold for $(\F^P_t)$-predictable and -optional processes.


We hope that this convinces the reader that there are no problems with using non-complete filtrations. Whenever we apply a result not from \cite{Jacod1979} or \cite{Jacod2003}, we point out why it also holds under non-complete filtrations.

After becoming aware of Rokhlin's work \cite{Rokhlin2010}, we noticed that our arguments prove in fact the main result of \cite{Rokhlin2010}, which is stronger than Theorem \ref{thm: ex supermartingale density}.

A family of nonnegative stochastic processes $\Y$ is called \emph{fork-convex}, see \cite{Zitkovic2002} or \cite{Rokhlin2010}, if every $Y \in \Y$ stays in zero once it hits zero, i.e. $Y_s = 0$ implies $Y_t = 0$ for all $0\le s \le t < \infty$, and if further for all $Y^1, Y^2, Y^3 \in \Y$, for all $s > 0$, and for all $\F_s$-measurable random variables $\lambda_s$ with values in $[0,1]$, we have that
\begin{align}\label{eq:fork-convex}
   Y_\cdot = 1_{[0,s)}(\cdot) Y^1_s + 1_{[s,\infty)}(\cdot) Y^1_s \left( \lambda_s \frac{Y^2_\cdot}{Y^2_s} + (1 - \lambda_s) \frac{Y^3_\cdot}{Y^3_s}\right) \in \Y.
\end{align}
Here and throughout the paper we interpret $0/0=0$. Note that a fork-convex family of processes with $Y_0 = 1$ for all $Y \in \Y$ is convex. If moreover $\Y$ contains the constant process 1, then $\Y$ is stable under stopping at deterministic times, i.e. for all $Y \in \Y$ and for all $t \ge 0$ also $Y_{\cdot \wedge t} \in \Y$.


Rokhlin's \cite{Rokhlin2010} main result is the following.

\begin{thm}\label{thm:abstract ex supermartingale density}
   Let $\Y$ be a fork-convex family of right-continuous and nonnegative processes containing the constant process 1 and such that $Y_0=1$ for all $Y \in \Y$. Let
   \begin{align*}
      \K = \left\{ Y_\infty: Y \in \Y, Y_\infty = \lim_{t\rightarrow \infty} Y_t \text{ exists}\right\}.
   \end{align*}
   Then $\K$ is bounded in probability if and only if there exists a supermartingale density for $\Y$.
\end{thm}

We split up the proof in several lemmas.

\begin{lem}\label{prop: l0-bded iff Z ex}
   Let $\X$ be a convex family of nonnegative random variables. Then $\X$ is bounded in probability if and only if there exists a strictly positive random variable $Z$ such that 
   \begin{align*}
      \sup_{X \in \X} E(X Z) < \infty.
   \end{align*}
\end{lem}

\begin{proof}
   The sufficiency is Theorem 1 of \cite{Yan1980}. Note that Yan does not require the $\sigma$-algebra to be complete. Yan makes the additional assumption that $\X$ is contained in $L^1$. But since we are considering nonnegative random variables, this can be avoided by applying Theorem 1 of \cite{Yan1980} to the convex hull of the bounded random variables $\{X \wedge n\}$ for $n \in \N$, as suggested in Remark (c) of \cite{Dellacherie1980}, VIII-84.
   
   So let us assume that $Z$ exists. Normalizing by $E(Z)$, we obtain an equivalent measure $Q$ such that $\X$ is norm bounded in $L^1(Q)$ and therefore bounded in $Q$-probability. Since $P\ll Q$, it is easy to see that $\X$ is also bounded in $P$-probability.
\end{proof}

\begin{rmk}
   Convexity is necessary. Let $\{A^n_k: 1 \le k \le 2^n, n \in \N\}$ be an increasing sequence of partitions of $\Omega$, such that for every $n,k$ we have $P(A^n_k) = 2^{-n}$. Define the nonnegative random variables $X^n_k = 1_{A^n_k} 2^{2n}$. Then $(X^n_k: n,k)$ is bounded in probability. Let $Z$ be a nonnegative random variable such that $E(Z X^n_k) \le C$ for some $C>0$ and all $n, k$. Then
   \begin{align*}
      E(1_{A^n_k} Z) = E(Z X^n_k) 2^{-2n} \le C 2^{-2n}.
   \end{align*}
   Summing over $k$, we obtain $E(Z) \le C 2^{-n}$ for all $n$, and therefore $E(Z) = 0$. Since $Z\ge 0$, this implies $Z=0$.
\end{rmk}

We call a family of random variables \emph{$L^p$-bounded} for $p\ge 1$ if it is norm bounded in $L^p$.

\begin{rmk}
   Lemma \ref{prop: l0-bded iff Z ex} states that a convex family of nonnegative random variables $\X$ is bounded in probability if and only if there exists a measure $Q \sim P$, such that $\X$ is $L^1(Q)$-bounded. One might ask if this can be improved. For example there could exist $Q\sim P$ such that $\X$ is $L^p(Q)$-bounded for some $p>1$. However this is not true in general. Even for a continuous martingale $M$ there might not be an absolutely continuous $Q \ll P$, such that $\K_1(M)$ (defined in terms of $M$ as in \eqref{eq:K1}) is uniformly integrable under $Q$. To see this, choose an increasing sequence of partitions $(A^n_k: 1 \le k \le 2^n, n \in \N)$ of $\R$, such that $\nu(A^n_k) = 2^{-n}$ for all $n,k$, where $\nu$ denotes the standard normal distribution. Let $M$ be a Brownian motion. Define the random variables $X^n_k = 1_{A^n_k}(M_1) 2^n$. Then $X^n_k \in L^\infty$, and $E(X^n_k) = 1$ for all $n,k$. By the predictable representation property of Brownian motion, $X^n_k \in \K_1(M)$ for all $n, k$. Now let $Q \ll P$, and let $g \ge 0$ be such that $\lim_{x \rightarrow \infty} g(x) / x = \infty$. If we show that $(g(X^n_k))_{n,k}$ is unbounded in $L^1(Q)$, then de la Vall\'ee-Poussin's theorem implies that $\K_1(M)$ cannot be uniformly integrable under $Q$. Let $C > 0$ and let $n \in \N$ be such that $g(2^n) \ge C 2^{n}$. Choose $k$ for which $Q( M_1 \in A^n_k) \ge 2^{-n}$. Such a $k$ must exist because $Q$ has total mass 1. Then
   \begin{align*}
      E_Q(g(X^n_k)) \ge E_Q(1_{A^n_k}(M_1) 2^n C) \ge 2^{-n} 2^n C = C.
   \end{align*}
   Since $C>0$ was arbitrary, this shows that $E_Q(g(\cdot))$ is unbounded on $\K_1(M)$.
\end{rmk}

The following Lemma establishes Theorem \ref{thm:abstract ex supermartingale density} in the case of two time steps. The general case then follows easily.

\begin{lem}\label{lem: one period supermartingale}
   Let $\Y$ be a $L^1$-bounded family of nonnegative processes indexed by $\{0,1\}$, adapted to a filtration $(\F_0, \F_1)$. Assume that $\Y$ is fork-convex and that $\Y$ contains a process of the form $(1, Y^*_1)$ for a strictly positive $Y^*_1$.

   Then there exists a strictly positive $\F_0$-measurable random variable $Z$, such that $(Y_0 Z, Y_1)$ is a supermartingale for every $Y \in \Y$.  $Z$ can be chosen such that for every $Y \in \Y$
   \begin{align}\label{eq: bound on one period supermartingale}
      \qquad E(Y_0 Z) \le \sup_{Y\in \Y} \max_{i = 0,1} E(Y_i).
   \end{align}
\end{lem}

\begin{proof}
   We define a nonnegative set function $\mu$ on $\F_0$ by setting
      $\mu(A) := \sup_{Y \in \Y} E( 1_A Y_1/Y_0)$.   
   Let us apply the fork-convexity of $\Y$ to show that for every $Y \in \Y$ there exists $\tilde{Y} \in \Y$, such that $Y_1/Y_0 = \tilde{Y}_1$. We take $s=0$ and $Y^1 = (1, Y^*_1)$ and $Y^2 = Y$ and $\lambda_s = 1$ in \eqref{eq:fork-convex}. Then $\tilde{Y} \in \Y$, where
   \begin{align*}
      \tilde{Y}_\cdot = 1_{\{0\}}(\cdot) + 1_{\{1\}}(\cdot) \frac{Y_1}{Y_0}.
   \end{align*}
   In particular we can use the $L^1$-boundedness of $\Y$ to obtain
   \begin{align*}
      \mu(A) = \sup_{Y \in \Y} E\left( 1_A \frac{Y_1}{Y_0}\right) \le \sup_{\tilde{Y} \in \Y} E(1_A \tilde{Y}_1) < \infty
   \end{align*}
   for all $A$, i.e. we obtain that $\mu$ is finite. In fact $\mu$ is a finite measure. Let $A,B \in \F_0$ be two disjoint sets and let $Y^A, Y^B \in \Y$. We take $s=0$, $Y^1 = (1, Y^*_1)$, $Y^2 = Y^A$, $Y^3 = Y^B$, and $\lambda_s = 1_A$ in \eqref{eq:fork-convex}, which implies $\tilde{Y} \in \Y$, where
   \begin{align*}
      \tilde{Y} =  1_{\{0\}}(\cdot) + 1_{\{1\}}(\cdot) \left(1_A \frac{Y^A_1}{Y^A_0} + 1_{A^c} \frac{Y^B_1}{Y^B_0}\right).
   \end{align*}
   Because $A$ and $B$ are disjoint, this $\tilde{Y}$ satisfies
   \begin{align*}
      1_{A \cup B} \frac{\tilde{Y}_1}{\tilde{Y}_0} = 1_A \frac{Y^A_1}{Y^A_0} + 1_{B} \frac{Y^B_1}{Y^B_0}.
   \end{align*}
   As a consequence we obtain
   \begin{align*}
      \mu(A) + \mu(B) & = \sup_{(Y^A, Y^B) \in \Y^2} E\left(1_A \frac{Y^A_1}{Y^A_0} + 1_B \frac{Y^B_1}{Y^B_0}\right) \le  \sup_{\tilde{Y} \in \Y} E\left(1_{A \cup B} \frac{\tilde{Y}_1}{\tilde{Y}_0}\right) = \mu(A \cup B).
   \end{align*}
   But $\mu(A \cup B) \le \mu(A) + \mu(B)$ is obvious, and therefore $\mu$ is finitely additive.
   
   Now let $(A_n)$ be a sequence of disjoint sets in $\F_0$. Then
   \begin{align*}
      \mu(\cup_{n=1}^\infty A_n) & =  \sup_{Y \in \Y} \sum_{n=1}^\infty E\left( 1_{A_n} \frac{Y_1}{Y_0}\right) \le  \sum_{n=1}^\infty \sup_{Y^n \in \Y} E\left( 1_{A_n} \frac{Y^n_1}{Y^n_0}\right) = \sum_{n=1}^\infty \mu(A_n).
   \end{align*}
   The opposite inequality is easily seen to be true for any finitely additive nonnegative set function. Thus $\mu$ is a finite measure on $\F_0$, which is absolutely continuous with respect to $P$. Therefore there exists a nonnegative $Z \in L^1(\F_0, P)$, such that
   \begin{align}\label{eq:radon proof1}
      \mu(A) = E(1_A Z) = \sup_{Y \in \Y} E\left( 1_A \frac{Y_1}{Y_0} \right).
   \end{align}
   It is easy to see that we can replace $1_A$ in \eqref{eq:radon proof1} by any nonnegative $\F_0$-measurable random variable. In particular for any $Y \in \Y$ and $A \in \F_0$
   \begin{align*}
      E(1_A Y_0 Z) = \sup_{\tilde{Y} \in \Y} E\left(1_A Y_0 \frac{\tilde{Y}_1}{\tilde{Y}_0}\right) \ge E\left(1_A Y_0 \frac{Y_1}{Y_0}\right) = E(1_A Y_1),
   \end{align*}
   proving that $(Y_0 Z, Y_1)$ is a supermartingale as long as $E(Y_0 Z) < \infty$. But the bound stated in \eqref{eq: bound on one period supermartingale} follows immediately from the fork-convexity of $\Y$, because the process $\bar{Y} = (Y_0,Y_0 \tilde{Y}_1/\tilde{Y}_0)$ is in $\Y$ for any $\tilde{Y} \in \Y$, and thus
   \begin{align*}
      E(Y_0 Z) = \sup_{\tilde{Y} \in \Y} E\left(Y_0 \frac{\tilde{Y}_1}{\tilde{Y}_0}\right) \le \sup_{\bar{Y} \in \Y} E\left(\bar{Y}_1\right).
   \end{align*}
   It remains to show that $Z$ is strictly positive. But this is easy, because $(1,Y^*_1)$ is in $\Y$, and $Y^*_1$ is strictly positive. Therefore $(Z,Y^*_1)$ is a supermartingale with strictly positive terminal value, which is only positive if also $Z$ is strictly positive.
\end{proof}

\begin{rmk}
   Some sort of stability assumption is necessary for Lemma \ref{lem: one period supermartingale} to hold. Even for a uniformly integrable and convex family of processes $\Y$, the proposition may fail without assuming fork convexity or a similar stability property. Let again $\{A^n_k: 1 \le k \le 2^n, n \in \N\}$ be an increasing sequence of partitions of $\Omega$, such that for every $n$ and $k$ we have $P(A^n_k) = 2^{-n}$. Define the random variables $Y^n_k = 1_{A^n_k} 2^n/n$. Let $M > 1$ be such that $2^{n_0-1} < M \le 2^{n_0}$. Then
   \begin{align*}
      \sup_{n,k} E(|Y^n_k| 1_{\{|Y^n_k| \ge M\}}) \le E(|Y^{n_0}_1|) = \frac{1}{n_0},
   \end{align*}
   proving that $(Y^n_k)_{n,k}$ is uniformly integrable. From de la Vall\'{e}e-Poussin's theorem and Jensen's inequality we obtain that also the convex hull $\mathcal{C}$ of the $Y^n_k$ is uniformly integrable. Define $\F_0 = \F_1 = \sigma(A^n_k: 1 \le k \le 2^n, n \in \N)$, and $\Y = \{(1,Y): Y \in \mathcal{C}\}$. Assume there exists $Z > 0$ such that $E(1_A Y) \le E(1_A Z)$ for all $A \in \F_0$ and $Y \in \mathcal{C}$. Then for every $n \in \N$
   \begin{align*}
      E(Z) = \sum_{k=1}^{2^n} E(1_{A^n_k} Z) \ge \sum_{k=1}^{2^n} E(1_{A^n_k} Y^n_k) = \frac{2^n}{n}.
   \end{align*}
   so that $E(Z) = \infty$. Therefore $(Z, Y)$ cannot be a supermartingale for any $Y \in \mathcal{C}$. In fact it is possible to show that $E(1_{A^n_k} Z) = \infty$ for all $k,n$, and since $\F_0$ is generated by $(A^n_k)_{n,k}$, this implies $P(Z = \infty) = 1$.
\end{rmk}

\begin{cor}\label{cor: discrete time supermartingale}
   Let $\Y$ be a $L^1$-bounded family of nonnegative processes indexed by $\{0, \dots, n\}$, adapted to a filtration $(\F_k: 0 \le k \le n)$. Assume that $\Y$ is fork-convex, and that it contains the constant process $(1, \dots, 1)$.
   
   Then there exists a strictly positive and adapted process $(Z_k: 0 \le k \le n)$, with $Z_n=1$ and such that $ZY$ is a supermartingale for every $Y \in \Y$. $Z$ can be chosen such that that every $Y \in \Y$
   \begin{align}\label{eq: bound on multiperiod supermartingale}
      \max_{k=0\dots, n} E(Z_k Y_k) \le \sup_{Y\in \Y} \max_{k = 0,\dots, n} E(Y_k).
   \end{align}
\end{cor}

\begin{proof}
   We prove the result by induction. For $n=1$, this is just Lemma \ref{lem: one period supermartingale}. Take $Z_0 = Z, Z_1 = 1$.

   Now assume the result holds for $n$. Let $\Y$ be a family of processes indexed by $\{0,\dots, n+1\}$, and assume that $\Y$ satisfies all requirements stated above. Then also $(Y_1, \dots, Y_{n+1})$ satisfies all those requirements. By induction hypothesis, there exists a strictly positive and adapted process $Z = (Z_1, \dots, Z_{n},1)$ such that $ZY$ is a supermartingale for all $Y\in \Y$, and such that
   \begin{align}\label{eq:multiperiod proof1}
      \max_{k=1\dots, n+1} E(Z_k Y_k) \le \sup_{Y\in \Y} \max_{k = 1,\dots, n+1} E(Y_k).
   \end{align}
   Therefore it suffices to construct a suitable $Z_0$. This is achieved by applying Lemma \ref{lem: one period supermartingale} to the family of processes $\tilde{\Y}=\{(Y_0, Z_1 Y_1): Y \in\Y\}$. Since $\Y$ contains the constant process $(1,\dots,1)$, $\tilde{\Y}$ contains the process $(1, Z_1)$, and $Z_1$ is strictly positive. Furthermore it is straightforward to check that $\tilde{Y}$ is fork-convex. $L^1$-boundedness of $\tilde{Y}$ follows from \eqref{eq:multiperiod proof1}. Hence we can apply Lemma \ref{lem: one period supermartingale} to $\tilde{Y}$, and the result follows.
\end{proof}

To prove Theorem \ref{thm:abstract ex supermartingale density}, we have to go from finite discrete time to continuous time. This is achieved by means of a compactness argument. Compactness for right-continuous functions is not very easy to show, and it would require us to use some form of the Arzel\`a-Ascoli theorem. However we want to construct a \emph{supermartingale} $Z$, and therefore it will be sufficient to construct its ``skeleton'' $(Z_q: q \in \Q_+)$. Using standard results for supermartingales, we can then use this skeleton to construct a right-continuous supermartingale density.

%

We will need the notion of convex compactness as introduced by Zitkovic \cite{Zitkovic2010}.
\begin{defn}
   Let $X$ be a topological vector space. A closed convex subset $C \subseteq X$ is called \emph{convexly compact} if for any family $\{F_\alpha: \alpha \in A\}$ of closed convex subsets of $C$, we can only have $\cap_{\alpha \in A} F_\alpha = \emptyset$ if there exist already finitely many $\alpha_1,\dots, \alpha_n \in A$ for which $\cap_{i=1}^n F_{\alpha_i} = \emptyset$.
\end{defn}

Recall that $L^0$ is the space of real valued random variables, equipped with the topology of convergence in probability. Zitkovic \cite{Zitkovic2010} then characterizes convexly compact sets of nonnegative elements of $L^0$.

\begin{lem}[Theorem 3.1 of \cite{Zitkovic2010}]\label{prop: zitkovic}
   Let $\X$ be a convex set of nonnegative random variables, closed with respect to convergence in probability. Then $\X$ is convexly compact in $L^0$ if and only if it is bounded in probability.
\end{lem}

Note that Zitkovic works on a complete probability space. But completeness is not used in the proof of Theorem 3.1. There is only one point in the proof where it is not immediately clear whether completeness of the $\sigma$-algebra is needed: when Lemma A1.2 of \cite{Delbaen1994} is applied. However this lemma is formulated for general probability spaces.

In Proposition \ref{prop: tychonoff result} we prove a Tychonoff theorem for countable families of convexly compact subsets of metric spaces. This will be used in the following proof.

\begin{lem}\label{lem: Q indexed supermartingale}
   Let $\Y$ and $\K$ be as in Theorem \ref{thm:abstract ex supermartingale density}. Then there exists a nonnegative supermartingale $(\bar{Z}_q)_{q \in \Q_+ \cup\{\infty\}}$ with $\bar{Z}_\infty > 0$, such that $(\bar{Z}_q Y_q)_{q \in \Q_+}$ is a supermartingale for all $Y \in \X$.
\end{lem}
\begin{proof}
   Recall that $\Y$ is convex, and therefore $\K$ is convex as well. 
   By Lemma \ref{prop: l0-bded iff Z ex} there exists $Q \sim P$ such that $\K$ is $L^1(Q)$-bounded. We define $C = \sup_{X \in \K} E_Q(X)$ and introduce the family of processes
   \begin{align*}
      \mathcal{C} = \{(Z_q)_{q \in \Q_+ \cup\{\infty\}}: Z_\infty = 1, Z_q \ge 0, Z_q \in \F_q \text{ and } E_Q(Z_q)\le C \text{ for all } q\}.
   \end{align*}
   By Lemma \ref{prop: zitkovic} and Proposition \ref{prop: tychonoff result}, this is a convexly compact set in $\prod_{q \in \Q_+ \cup \{ \infty \}} L^0(\F_q, Q)$ equipped with the product topology (where every $L^0(\F_q, Q)$ is equipped with its usual topology). Define for given $q,r \in \Q_+ \cup\{\infty\}$
   \begin{align*}
      \mathcal{C}(q,r) = \{Z \in \mathcal{C}: E_Q(Z_{q+r} Y_{q+r}/Y_q|\F_q) \le Z_q \text{ for all } Y \in \Y\},
   \end{align*}
   where for $r=\infty$ it is understood that we only consider those $Y \in \Y$ for which $Y_\infty = \lim_{t\rightarrow \infty} Y_t$ exists. The sets $\mathcal{C}(q,r)$ are convex, and by Fatou's lemma they are also closed. Furthermore they are subsets of the convexly compact set $\mathcal{C}$.  
   So if
   \begin{align*}
      \cap_{q \in \Q_+, r \in \Q_+ \cup \{\infty\}}\mathcal{C}(q,r)
   \end{align*}
   was empty, then already a finite intersection would have to be empty. But if a finite intersection was empty, then there would be some $0 \le t_0< \dots < t_n \le \infty$ for which it is impossible to find $(Z_{t_i}: i=0,\dots, n)$ with $Z_{t_n}=1$ and such that $(Y_{t_i} Z_{t_i}: i = 0,\dots, n)$ is a $Q$-supermartingale for every $Y \in \Y$ (respectively for every $Y \in \Y$ for which $\lim_{t \rightarrow \infty} Y_t$ exists in case $t_n = \infty$). This would contradict Corollary \ref{cor: discrete time supermartingale}.

   So let $Z$ be in the intersection of all $\mathcal{C}(q,r)$ and let $Y \in \Y$. Then for any $q, r \in \Q_+$
   \begin{align*}
      E_Q\left(\left.Z_{q+r} \frac{Y_{q+r}}{Y_q}\right|\F_q\right) \le  Z_q,
   \end{align*}
   which shows that $ZY$ is a $Q$-supermartingale on $\Q_+$. Taking $Y\equiv 1$, we also see that $(Z_q: q \in \Q_+ \cup \{\infty\})$ is a supermartingale. To complete the proof it suffices now to define for $q \in \Q_+ \cup \{\infty\}$
   \begin{align*}
      \bar{Z}_q = Z_q \left.\frac{dP}{dQ}\right|_{\F_q}.
   \end{align*}
\end{proof}

We are now ready to prove Rokhlin's result.

\begin{proof}[Proof of Theorem \ref{thm:abstract ex supermartingale density}]
   It remains to show that given the skeleton $(\bar{Z}_q: q \in \Q_+ \cup \{\infty\})$ we can construct a right-continuous supermartingale density with left limits almost everywhere. This is a standard result on supermartingales. For the reader's convenience and to dispel possible concerns about the non-completeness of our filtration, we give the arguments below.
   
   Since $(\F_t)$ is right-continuous, for every $t\ge 0$ there exists a null set $N_t \in \F_t$, such that for $\omega \in \Omega \backslash N_t$
   \begin{align*}
      \lim_{\substack{r\rightarrow s-\\r \in \Q}} \bar{Z}(\omega)_r \text{ and } \lim_{\substack{r\rightarrow s+\\r \in \Q}} \bar{Z}(\omega)_r
   \end{align*}
   exist for every $s \le t$. See for example Ethier and Kurtz \cite{Ethier1986}, right before Proposition 2.2.9. We define for $t \in [0,\infty)$
   \begin{align*}
      Z_t(\omega) = \begin{cases}
                       \lim_{\substack{s\rightarrow t+\\s \in \Q}} \bar{Z}_s(\omega), & \omega \in \Omega\backslash N_t \\
                       0, & \text{otherwise}.
                    \end{cases}
   \end{align*}
   Then $Z$ is adapted because $(\F_t)$ is right-continuous, and $Z$ is right-continuous by definition. However it may not have left limits everywhere. But outside of the null set $N = \cup_{n \in \N} N_n$ it has left limits at every $t > 0$.
   
   Let us show that $ZY$ is a supermartingale for every $Y \in \Y$. Recall that the processes in $\Y$ are right-continuous. Using Fatou's Lemma in the first step and Corollary 2.2.10 of \cite{Ethier1986} in the second step, we obtain
   \begin{align*}
      E_Q(Z_{t+s} Y_{t+s}|\F_t) &\le \liminf_{\substack{r\rightarrow (t+s)+\\r \in \Q}} E_Q(\bar{Z}_r Y_r |\F_t) = \liminf_{\substack{r\rightarrow (t+s)+\\r \in \Q}} \liminf_{\substack{u\rightarrow t+\\u \in \Q}} E_Q(\bar{Z}_r Y_r |\F_u) \\
      &\le \liminf_{\substack{r\rightarrow (t+s)+\\r \in \Q}} \liminf_{\substack{u\rightarrow t+\\u \in \Q}} \bar{Z}_u Y_u = Z_t Y_t.
   \end{align*}
   The same arguments with $s=\infty$ and $Y=1$ show that if we set $Z_\infty = \bar{Z}_\infty$, then $(Z_t)_{t\in [0,\infty]}$ is a nonnegative supermartingale with strictly positive terminal value. By Theorem I.1.39 of \cite{Jacod2003}, $Z_t$ almost surely converges to a limit $\tilde{Z}_\infty$ as $t\rightarrow \infty$. Theorem VI-6 of \cite{Dellacherie1980} now implies $\tilde{Z}_\infty \ge Z_\infty > 0$. Since \cite{Dellacherie1980} work with complete filtrations, let us provide the argument. Define $M_t = E(Z_\infty |\F_t)$ for $t \in [0,\infty]$. This is a uniformly integrable martingale which almost surely converges to $Z_\infty$ as $t \rightarrow \infty$. By the supermartingale property of $Z$ we have $M_t \le Z_t$ for all $t \ge 0$. Thus
   \begin{align*}
      Z_\infty = \lim_{t\rightarrow \infty} M_t \le \lim_{t \rightarrow \infty} Z_t = \tilde{Z}_\infty.
   \end{align*}
   
   It remains to show necessity. If $Z$ is a supermartingale density for $\Y$, then for any $Y \in \Y$, $Z_t Y_t$ converges as $t\rightarrow \infty$, see Theorem I.1.39 of \cite{Jacod2003}. Since $Z_t$ converges to a strictly positive limit, $Y_t$ must converge as well, and we have $E(Z_\infty Y_\infty) \le E(Z_0 Y_0) = E(Z_0)$. Now Lemma \ref{prop: l0-bded iff Z ex} implies that $\Y$ is bounded in probability.
\end{proof}

\begin{cor}\label{cor:abstract semimartingale}
   If $\Y$ is as in Theorem \ref{thm:abstract ex supermartingale density}, then every $Y \in \Y$ is a semimartingale for which $Y_t$ almost surely converges as $t \rightarrow \infty$.
\end{cor}
\begin{proof}
   Convergence was shown in the proof of Theorem \ref{thm:abstract ex supermartingale density}. The semimartingale property follows from It\^{o}'s formula. Let $Z$ be a supermartingale density for $\Y$. Then $Z$ is strictly positive, and therefore $1/Z$ is a semimartingale, implying that $Y = (1/Z) (ZY)$ is a semimartingale.
\end{proof}
In case $\Y = \W_1$ and under the stronger assumption (NFLVR$_s$), Corollary \ref{cor:abstract semimartingale} was already shown in \cite{Delbaen1994}.

\begin{proof}[Proof of Theorem \ref{thm: ex supermartingale density}]
   It suffices to note that $\W_1$ and $\W_{1,s}$ satisfy the assumptions of Theorem \ref{thm:abstract ex supermartingale density}. This is easy and shown for example in Rokhlin \cite{Rokhlin2010}, in the proof of Theorem 2. Rokhlin only treats the case of $\W_1$ and $\K_1$, but the same arguments also work for $\W_{1,s}$ and $\K_{1,s}$.
\end{proof}

\begin{proof}[Proof of Corollary \ref{cor:NA1 implies semimartingale}]
   Let $S$ have components that are locally bounded from below and assume that $S$ satisfies (NA1$_s$). Recall that local semimartingales are semimartingales, see for example Protter \cite{Protter2004}, Theorem II.6. Protter works with complete filtrations, but it is an immediate consequence of Lemma \ref{lem:complete cadlag} in Appendix \ref{app:complete filtration} that for every $(\F^P_t)$-semimartingale there exists an indistinguishable $(\F_t)$-semimartingale. Let $1 \le i \le d$. Since $S^i$ is locally bounded from below, there exists an increasing sequence of stopping times $(T_n)$ with $\lim_{n \rightarrow \infty} T_n = \infty$, and a sequence of strictly positive numbers $\alpha_n$, such that $(1 + \alpha_n S^i_{t \wedge T_n})_{t\ge0} \in \W_{1,s}$. By Corollary \ref{cor:abstract semimartingale} the stopped process $S^i_{\cdot\wedge T_n}$ is a semimartingale, and therefore $S^i$ is a semimartingale.
   
   It remains to show that in the locally bounded from below case, any supermartingale density for $\W_{1,s}$ is a supermartingale density for $\W_1$. But this is exactly Kardaras and Platen \cite{Kardaras2011a}, Section 2.2. (And it will also follow from our considerations in Section \ref{sec: dominating measures}.)
\end{proof}

Of course we could also assume that every component of $S$ is either locally bounded from below or locally bounded from above, and we would still obtain the semimartingale property of $S$ under (NA1$_s$). But in the totally unbounded case, $S$ is not necessarily a semimartingale. A simple counterexample is given by a one-dimensional L\'{e}vy-process with jumps that are unbounded both from above and from below, to which we add an independent fractional Brownian motion with Hurst index $H \neq 1/2$. Their sum is not a semimartingale. But there are no 1-admissible strategies other than 0, so that $\K_{1,s} = \{1\}$, which is obviously bounded in probability.

\section{Construction of dominating local martingale measures}
\label{sec: dominating measures}

\subsection{The Kunita-Yoeurp problem and F\"ollmer's measure}\label{subsec: Kunita-Yoeurp}
Now let $Z$ be a strictly positive supermartingale with $Z_\infty >0$ and $E_P(Z_0) = 1$. Here we construct a dominating measure associated to $Z$. Naturally it is more delicate to construct a dominating measures than to construct an absolutely continuous measure.

Our aim is to construct a dominating measure $Q$ and a stopping time $T$, such that $(Z,T)$ is the Kunita-Yoeurp decomposition of $Q$ with respect to $P$. We call this the ``Kunita-Yoeurp problem''. Recall that $(Z,T)$ is the Kunita-Yoeurp decomposition of $Q$ with respect to $P$ if
\begin{enumerate}
   \item $P(T = \infty) = 1$,
   \item $Q(\cdot \cap \{T \le t\})$ and $P$ are mutually singular on $\F_t$,
   \item $Q(\cdot \cap \{T > t\})$ is absolutely continuous with respect to $P$ on $\F_t$, and for $A \in \F_t$
      \begin{align}\label{eq: kunita decomposition}
         Q(A \cap \{T>t\}) = E_P(1_A Z_t).
      \end{align}
\end{enumerate}

In this case one can show that for any stopping time $\tau$ and any $A \in \F_\tau$
\begin{align}\label{eq: kunita decomposition stopping time}
   Q(A \cap \{T > \tau\}) = E_P\left(1_{A \cap \{\tau < \infty\}} Z_\tau\right),
\end{align}
see for example \cite{Yoeurp1985}, Proposition 4.

In general it is impossible to construct $Q$ and $T$ without making further assumptions on the underlying filtered probability space. For example the space could be too small. Take an $\Omega$ that consists of one element, and define $\F = \F_t = \{\emptyset, \Omega\}$ for all $t \ge 0$. Then 
\begin{align*}
   Z_t = \frac{1}{2}(1+e^{-t}), \qquad t \ge 0,
\end{align*}
is a continuous and positive supermartingale with $Z_\infty >0$. But there exists only one probability measure on $\Omega$, and therefore any $Q$ would have Kunita-Yoeurp decomposition $(1,\infty)$ with respect to $P$, and not $(Z,T)$. This is reminiscent of the Dambis Dubins-Schwarz Theorem without the assumption $\langle M \rangle_\infty = \infty$ (see Revuz and Yor \cite{Revuz1999}, Theorem V.1.7). This problem can be solved by enlarging $\Omega$.

But even if the space is large enough, it might still not be possible to find $Q$ and $T$, because the filtration might be too large. Assume that the filtration $(\F_t)$ is complete with respect to $P$ and that $E_P(Z_0) = 1$. Then \eqref{eq: kunita decomposition} implies that $Q$ is absolutely continuous with respect to $P$ on $\F_0$. Since $\F_0$ contains all $P$-null sets, this means that $Q$ is absolutely continuous with respect to $P$, and therefore by \eqref{eq: kunita decomposition} again $Z_t = E_P(dQ/dP|\F_t)$. That is, $Z$ has to be a uniformly integrable martingale under $P$. So if $Z$ is a supermartingale, then the filtration $(\F_t)$ should not be completed. This problem can be avoided by assuming that $(\F_t)$ is the right-continuous modification of a standard system, to be defined below.

If we are allowed to enlarge $\Omega$ and if $(\F_t)$ is the right-continuous modification of a standard system, then the problem of constructing $Q$ and $T$ has been solved by Yoeurp \cite{Yoeurp1985} with the help of F\"ollmer's measure. Let us describe Yoeurp's solution.

First we remove the second problem in constructing $Q$ and $T$ by assuming that the filtration $(\F_t)$ is the right-continuous modification of a standard system $(\F^0_t)$. A filtration $(\F^0_t)$ is called \emph{standard system} if
\begin{enumerate}
   \item for all $t \ge 0$, the $\sigma$-algebra $\F^0_t$ is $\sigma$-isomorphic to the Borel $\sigma$-algebra of a Polish space; that is, there exists a Polish space $(\X_t, \B_t)$, and a bijective map $\pi: \F^0_t \rightarrow \B_t$, such that $\pi^{-1}$ preserves countable set operations;
   \item if $(t_i)_{i \ge 0}$ is an increasing sequence of positive times, and if $(A_i)_{i \ge 0}$ is a decreasing sequence of atoms of $\F^0_{t_i}$, then $\cap_{i \ge 0} A_i \neq \emptyset$.
\end{enumerate}
Then $(\F_t)$ is defined by setting $\F_t = \cap_{s > t} \F^0_{s}$.
Path spaces equipped with the canonical filtration are only standard systems if we allow for ``explosion'' to a cemetery state in finite time, see F\"ollmer \cite{Follmer1972}, Example 6.3, 2), or Meyer \cite{Meyer1972}. 

Next we enlarge $\Omega$ in order to solve the possible problem of $\Omega$ being too small. Define $\overline{\Omega} := \Omega \times [0, \infty]$ and $\overline{\F} = \F \otimes \B[0,\infty]$, where $\B[0,\infty]$ denotes the Borel $\sigma$-algebra of $[0,\infty]$. Also define $\overline{P} = P \otimes \delta_\infty$, where $\delta_\infty$ is the Dirac measure at $\infty$. The filtration $(\overline{\F}_t)$ is defined by
\begin{align*}
   \overline{\F}_t = \cap_{s > t} \F_s \otimes \sigma([0,r]: r \le s).
\end{align*}
Note that if $(\F_t)$ is the right-continuous modification of the standard system $(\F^0_t)$, then $(\overline{\F}_t)$ is the right-continuous modification of the standard system $\overline{\F}^0_0 = \F^0_0 \otimes \sigma(\{0\})$,
\begin{align*}
   \overline{\F}^0_t = \F^0_t \otimes \sigma([0,s]: s \le t), \qquad t > 0.
\end{align*}
Random variables $X$ on $\Omega$ are embedded into $\overline{\Omega}$ by setting $\overline{X}(\omega, \zeta) = X(\omega)$.

Let us remark that $(\overline{\Omega}, \overline{\F}, (\overline{\F}_t), \overline{P})$ is an enlargement of $(\Omega, \F, (\F_t), P)$ in the sense of \cite{Revuz1999}.

\begin{defn}[\cite{Revuz1999}, p. 182] 
   A filtered probability space $(\tilde{\Omega}, \tilde{\F}, (\tilde{\F}_t), \tilde{P})$ is an \emph{enlargement} of $(\Omega, \F, (\F_t), P)$ if there exists a measurable map $\pi: \tilde{\Omega} \rightarrow \Omega$, such that $\pi^{-1}(\F_t) \subseteq \tilde{\F}_t$ and such that $\tilde{P}\circ \pi^{-1} = P$. In this case, random variables are embedded from $(\Omega, \F)$ into $(\tilde{\Omega}, \tilde{\F})$ by setting $\tilde{X}(\tilde{\omega}) = X(\pi(\tilde{\omega}))$.
\end{defn}

Define $\pi(\omega, \zeta) = \omega$. Then we have $\pi^{-1}(A) = A \times [0, \infty] \in \overline{\F_t}$ for every $A \in \F_t$, and therefore $\pi^{-1}(\F_t) \subseteq \overline{\F_t}$. For any set $A \in \F$ we have $\overline{P}\circ \pi^{-1}(A) = P \otimes \delta_\infty (A \times [0, \infty]) = P(A)$. And if $X$ is a random variable on $(\Omega, \F)$, then $\overline{X}(\omega, \zeta) = X(\omega) = X(\pi(\omega, \zeta))$.

Now we can proceed to construct $(\overline{Q},\overline{T})$ on $(\overline{\Omega}, \overline{\F}, (\overline{\F}_t))$. In fact it suffices to construct $\overline{Q}$, because we will take $\overline{T}(\omega, \zeta) = \zeta$, so that $\overline{P}(\overline{T} = \infty) = 1$. However there is one last remaining problem. In general $\overline{Q}$ will not be uniquely determined by $(\overline{Z},\overline{T})$. The measure $\overline{Q}$ must satisfy
\begin{align}\label{eq:martingale iff T infinite}
   1 = \overline{Q}(\overline{\Omega}) = \overline{Q}(\overline{\Omega} \cap\{ t < \overline{T}\}) + \overline{Q}(\overline{\Omega} \cap\{ t \ge \overline{T}\}) = E_{\overline{P}}(\overline{Z}_t) + \overline{Q}( t \ge \overline{T})
\end{align}
for all $t \ge 0$. But $\overline{Q}$ is supposed to solve the Kunita-Yoeurp problem associated with $(\overline{Z},\overline{T})$, and therefore at time $\overline{T}$, the measure $\overline{Q}$ should stop being absolutely continuous with respect to $\overline{P}$, and \eqref{eq:martingale iff T infinite} implies that $\overline{Q}(\overline{T}<\infty) > 0$ if $\overline{Z}$ is not a martingale. So knowing $\overline{P}$, $\overline{Z}$, and $\overline{T}$, in general we can only hope to determine $\overline{Q}$ uniquely on the $\sigma$-field
\begin{align}\label{eq:F T minus predictable}
   \overline{\F}_{\overline{T}-} & = \sigma(\overline{\F}_0,\overline{A}_t \cap \{{\overline{T} > t}\}: \overline{A}_t \in \overline{\F}_t, t > 0)  = \sigma(A_t \times (t, \infty]: A_t \in \F_t, t \ge 0).
\end{align}
For the second equality we refer to \cite{Follmer1972}. Note that \eqref{eq:F T minus predictable} implies that $\overline{\F}_{\overline{T}-}$ is the predictable sigma-algebra on $\Omega \times [0,\infty]$.

In conclusion, in order to construct $T$ and $Q$, we need to assume that $(\F_t)$ is the the right-continuous modification of a standard system, we need to enlarge $(\Omega, \F, (\F_t), P)$ as described above, and we have to accept that $Q$ will only be defined on $\overline{\F}_{\overline{T}-}$. Under these conditions, we can take $\overline{Q}$ as the  F\"ollmer measure of $Z$. Given a nonnegative supermartingale $Z$ with $E_P(Z_0) = 1$, F\"ollmer \cite{Follmer1972} constructs a measure $P^Z$ on $(\overline{\Omega}, \overline{\F}_{T-})$, which satisfies $P^Z(\overline{A}_t \cap \{ \overline{T} > t\}) = E_{\overline{P}}(\overline{Z}_t 1_{\overline{A}_t})$ for all $t \ge 0$ and all $\overline{A}_t \in \overline{\F}_t$. This is exactly the relation \eqref{eq: kunita decomposition}, and therefore $(\overline{Z}, \overline{T})$ is the Kunita-Yoeurp decomposition of $\overline{Q}:=P^Z$ with respect to $\overline{P}$. Note that it is possible to extend $\overline{Q}$ from $\overline{\F}_{\overline{T}-}$ to $\overline{\F}$, generally in a non-unique way. See for example the discussion on p. 9 of \cite{Kardaras2011}. So from now on we make the convention that the measure $\overline{Q}$ is one of these extensions, i.e. $\overline{Q}$ denotes a probability measure on $(\overline{\Omega}, \overline{\F})$ that satisfies $\overline{Q}|_{\overline{\F}_{\overline{T}-}} = P^Z$.

It remains to show that $\overline{Q}$ dominates $\overline{P}$. But this is a general result.

\begin{lem}\label{lem:absolute continuity}
   Let $\mu$ and $\nu$ be two probability measures on a filtered probability space $(\Theta, \G, (\G_t))$. Let $(\zeta, \tau)$ be the Kunita-Yoeurp decomposition of $\nu$ with respect to $\mu$. Define $\G_\infty = \vee_{t\ge 0} \G_t$ and assume that $\mu(\zeta_\infty > 0)=1$. Then $\mu|_{\G_{\infty}} \ll \nu|_{\G_{\infty}}$.
\end{lem}

\begin{proof}
   Let $A \in \cup_{t\ge0}\G_t$. We use the $\sigma$-continuity of $\nu$, \eqref{eq: kunita decomposition}, and Fatou's lemma, to obtain
   \begin{align*}
      \nu(A\cap\{\tau=\infty\})=\lim_{t\rightarrow\infty} \nu(A\cap\{\tau>t\}) = \lim_{t\rightarrow \infty} E_\mu(\zeta_t 1_{A}) \ge E_\mu(\zeta_\infty 1_{A}).
   \end{align*}
   By the monotone class theorem, this inequality extends to all $A \in \G_\infty$. Since $\mu(\zeta_\infty>0)=1$, we obtain $\mu|_{\G_\infty} \ll \nu|_{\G_\infty}$.
\end{proof}


Lemma \ref{lem:absolute continuity} applied to $\overline{P}$, $\overline{Q}$, and $(\overline{\Omega}, \overline{\F}, (\overline{\F}_t))$ implies that $\overline{Q}\gg \overline{P}$.

In the following we make the standing assumption that we work on a probability space $(\Omega, \F, (\F_t), P)$ where it is possible to solve the Kunita-Yoeurp problem of associating a probability measure and a stopping time to a given supermartingale, and we omit the notation $\overline{(
\cdot)}$. 

\subsubsection{Calculating expectations under $Q$}

Here we present important results of Yoeurp that allow to rewrite certain expectations under $Q$ as expectations under $P$. More precisely, let $Z$ be a nonnegative supermartingale with $E(Z_0)=1$ and with Doob-Meyer decomposition $Z = Z_0 + M - A$, where $M$ is a local martingale starting in zero, and $A$ is an adapted process, a.s. increasing and c\`adl\`ag. Let $T$ and $Q$ be a stopping time and a probability measure, such that $(Z,T)$ is the Kunita-Yoeurp decomposition of $Q$ wrt $P$. 

\begin{lem}\label{lem:yoeurps formula}
   Let $Z = Z_0 + M - A$, and let $T, Q$ be as described above. Let $(\tau_n)$ be a localizing sequence for $M$. Then we have for every bounded predictable process $Y$ and for every $n \in \N$
\begin{align}\label{eq:explicit foellmer}
   E_Q(Y^{\tau_n}_T) = E_P\left(Y_{\tau_n} Z_{\tau_n} + \int_0^{\tau_n} Y_s dA_s\right).
\end{align}
\end{lem}

\begin{proof}
   This is part of Proposition 9 of \cite{Yoeurp1985}. For the convenience of the reader, we provide a proof. First consider a simple processes of the form $Y_r(\omega) = X(\omega) 1_{(t,\infty]}(r)$ for some bounded $X \in \F_t$. For such $Y$ we get from \eqref{eq: kunita decomposition}
   \begin{align*}
      E_Q(Y^{\tau_n}_T) & = E_Q( X 1_{(t,\infty]} (\tau_n \wedge T)) = E_Q( X 1_{\{t < \tau_n\}} 1_{\{t < T\}}) = E_P( X 1_{\{t < \tau_n\}} Z_t) \\
      & = E_P( X 1_{\{t < \tau_n\}} (M_t - A_t)) = E_P( X 1_{\{t < \tau_n\}} (M^{\tau_n}_t - A_t)).
   \end{align*}
   Now we use that $M^{\tau_n}$ is a uniformly integrable martingale, and that $X 1_{\{t<\tau_n\}}$ is $\F_{t}$-measurable, to replace $M^{\tau_n}_t$ by $M^{\tau_n}_\infty = M_{\tau_n}$. Moreover we have
   \begin{align*}
      X 1_{\{t < \tau_n\}} A_t & = X 1_{\{t < \tau_n\}} A_{\tau_n} - X 1_{\{t < \tau_n\}}(A_{\tau_n} - A_t) = X 1_{\{t < \tau_n\}} A_{\tau_n} - \int_0^{\tau_n} Y_s dA_s,
   \end{align*}
   which proves \eqref{eq:explicit foellmer} for such simple $Y$. The general case now follows from the monotone class theorem. 
\end{proof}

\begin{cor}\label{cor:yoeurps formula 2}
   Let $Y$ be a bounded adapted process that is $P$-a.s. c\`adl\`ag. Define
   \begin{align*}
      Y^{T-}_t(\omega) = Y_t(\omega) 1_{\{t < T(\omega)\}} + \limsup_{s \rightarrow T(\omega)-} Y_s(\omega) 1_{\{t \ge T(\omega) \}}.
   \end{align*}
   Let $Z$ and $(\tau_n)$ be as in Lemma \ref{lem:yoeurps formula}. Then
   \begin{align*}
      E_Q(Y^{T-}_{\tau_n}) = E_P\left( Y_{\tau_n} Z_{\tau_n} + \int_0^{\tau_n} Y_{s-} dA_s\right).
   \end{align*}
\end{cor}

\begin{proof}
   Define $Y^{-}_t(\omega) = Y_{t-}(\omega) = \limsup_{s \rightarrow t-} Y_s$ for $t > 0$, and $Y^-_0 = Y_0$. Then $Y^-$ is a predictable process, because it is the point-wise limit of the step functions
   \begin{align*}
      Y^n_t = Y_0 1_{\{0\}}(t) + \sum_{k \ge 0} \limsup_{s \rightarrow k2^{-n} -} 1_{(k2^{-n}, (k+1)2^{-n}]}(t).
   \end{align*}
   Therefore we can apply Lemma \ref{lem:yoeurps formula} to $Y^-$. Observe that 
   \begin{align*}
      Y^{T-}_{\tau_n} = Y_{\tau_n} 1_{\{T > \tau_n\}} + Y_{T-} 1_{\{T \le \tau_n\}} = Y_{\tau_n} 1_{\{T > \tau_n\}} + (Y^-)_T 1_{\{T \le \tau_n\}}.
   \end{align*}
   Now \eqref{eq: kunita decomposition} implies that $E_Q(Y_{\tau_n} 1_{\{T > \tau_n\}}) = E_P(Y_{\tau_n} Z_{\tau_n})$, whereas \eqref{eq:explicit foellmer} and then again \eqref{eq: kunita decomposition} applied to the second term give
   \begin{align*}
      E_Q((Y^-)_T 1_{\{T \le \tau_n\}}) &=  E_Q((Y^-)^{\tau_n}_T ) - E_Q((Y^-)_{\tau_n} 1_{\{T > \tau_n\}}) \\
      &= E_P\left(Y_{\tau_n-} Z_{\tau_n} + \int_0^{\tau_n} Y_{s-} dA_s\right) - E_P\left(Y_{\tau_n-} Z_{\tau_n}\right) \\
      & = E_P\left( \int_0^{\tau_n} Y_{s-} dA_s\right).
   \end{align*}
\end{proof}

\subsection{The predictable case}

We still assume that $(\Omega, \F, (\F_t), P)$ is a probability space on which it is possible to solve the Kunita-Yoeurp problem. Let $S$ be a $d$-dimensional predictable semimartingale, let $\W_{1}$ be defined as in \eqref{eq:W1}, and let $Z$ be a supermartingale density for $\W_1$. Here we examine the structure of $S$ and $Z$ closer. This will allow us to apply Lemma \ref{lem:yoeurps formula} to deduce that $S^{T-}$ is a local martingale under the dominating measure associated to $Z$. Note that Yoeurp \cite{Yoeurp1985} also establishes a generalized Girsanov formula, which we could apply directly rather than using Lemma \ref{lem:yoeurps formula}. However the use of Lemma \ref{lem:yoeurps formula} turns out to be rather instructive, and it allows us to obtain some insight into why the non-predictable case is more complicated.

\begin{rmk}
   Observe that thanks to predictability, $S-S_0$ is almost surely locally bounded. This follows immediately from I.2.16 of \cite{Jacod2003}, which says that for $C>0$ there exists an announcing sequence for the entrance time of $S$ into $\{x \in \R^d: |x|\ge C\}$. By Corollary \ref{cor:NA1 implies semimartingale} it would in particular suffice to assume that $S$ satisfies (NA1$_s$) and that $Z$ is a supermartingale density for $\W_{1,s}$. Then $S$ is a semimartingale, satisfies (NA1), and $Z$ is a supermartingale density for $\W_1$.
\end{rmk}

Since $S-S_0$ is locally bounded, it is even a special semimartingale (see \cite{Jacod2003}, I.4.23 (iv)). That is, there exists a unique decomposition
\begin{align}\label{eq: semimartingale decomposition}
   S = S_0 + M + A,
\end{align}
where $M$ is a local martingale with $M_0 = 0$, and $A$ is a predictable process of finite variation with $A_0 = 0$. This implies that $M = S - S_0 - A$ is predictable. But any predictable right-continuous local martingale is continuous (\cite{Jacod2003}, Corollary I.2.31). Therefore $S$ is of the form \eqref{eq: semimartingale decomposition} with continuous $M$.

But then also $A$ must be continuous, because (NA1) implies $dA^i \ll d\langle M^i \rangle$ for every $i=1,\dots,d$, where $M=(M^1, \dots, M^d)$ and $A=(A^1,\dots, A^d)$. This is a well known fact, see for example Ankirchner's thesis \cite{Ankirchner2005}, Lemma 9.1.2. Otherwise one could find a predictable process $H^i$ which satisfies $H^i \cdot M^i \equiv 0$, but for which $H^i \cdot A^i$ is increasing; 
this would clearly contradict $\K_1$ being bounded in probability. Therefore $A$ and then also $S$ must be continuous.

It turns out that $S$ must satisfy the \emph{structure condition} as defined by Schweizer \cite{Schweizer1995}. Recall that $L^2_{\text{loc}}(M)$ is the space of progressively measurable processes $(\lambda_t)_{t \ge 0}$ that are locally square integrable with respect to $M$, i.e. such that
\begin{align*}
   \int_0^t \sum_{i,j=1}^d \lambda^i_s \lambda^j_s d \langle M^i, M^j \rangle_s < \infty
\end{align*}
for every $t >0$. For details see \cite{Jacod2003}, III.4.3.

\begin{defn}\label{def:structure condition}
   Let $S = S_0 + M + A$ be a $d$-dimensional special semimartingale with locally square-integrable $M$. Define
   \begin{align*}
      C_t = \sum_{i=1}^d \langle M^i \rangle_t \qquad \text{and for } 1 \le i, j \le d: \qquad \sigma^{ij}_t = \frac{d \langle M^i, M^j \rangle_t }{d C_t}.
   \end{align*}
   Note that $\sigma$ exists by the Kunita-Watanabe inequality. Then $S$ satisfies the \emph{structure condition} if $dA^i \ll d \langle M^i \rangle$ for every $1 \le i \le d$, with predictable derivative $\alpha^i_t = d A^i_t/d\langle M^i \rangle_t$, and if there exists a predictable process $\lambda_t = (\lambda^1_t, \dots, \lambda^d_t) \in L^2_{\text{loc}}(M)$, such that 
   \begin{align} \label{eq:structure condition}
      (\sigma_t \lambda_t)^i = \alpha^i_t \sigma^{ii}_t
   \end{align}
   for every $i = 1, \dots, d$. Note that $\lambda$ might not be uniquely determined, but the stochastic integral $\int \lambda dM$ does not depend on the choice of $\lambda$, see \cite{Schweizer1995}. If
   \begin{align}\label{eq:structure condition until infty}
      \int_0^\infty \sum_{i,j=1}^d \lambda^i_t \sigma^{ij}_t \lambda^j_t dC_t < \infty,
   \end{align}
   then we say that S satisfies the \emph{structure condition until $\infty$}.
\end{defn}

%

Recall that two one-dimensional local martingales $L$ and $N$ are called \emph{strongly orthogonal} if $LN$ is a local martingale. 
Also recall that the \emph{stochastic exponential} of a semimartingale $X$ is defined by
\begin{align*}
   \mathcal{E}(X)_t = 1 + \int_0^t \mathcal{E}(X)_{s-} dX_s, \qquad t \ge 0.
\end{align*}
Finally we recall that every nonnegative supermartingale $Y$ satisfies $Y_{\tau+t} \equiv 0$ for all $t \ge 0$, where $\tau = \inf\{t \ge 0: Y_{t-}=0 \text{ or } Y_t = 0\}$.

\begin{lem}\label{lem: structure condition for predictable S}
   Suppose that $Z$ is a supermartingale density for the predictable semimartingale $S$. Then $S$ satisfies the structure condition until $\infty$, and 
   \begin{align}\label{eq: explicit supermartingale density}
      dZ_t = Z_{t-}( - \lambda_t dM_t + dN_t - dB_t)
   \end{align}
   where $\lambda$ satisfies \eqref{eq:structure condition} and \eqref{eq:structure condition until infty}, $N$ is a local martingale that is strongly orthogonal to $M$, $B$ is increasing, 
   and $\mathcal{E}(N-B)_\infty > 0$.
   
   Conversely, if a predictable process $S$ satisfies the structure condition until $\infty$, and if $Z$ is defined by \eqref{eq: explicit supermartingale density}, then $Z$ is a supermartingale density for $S$.
   
   In particular, for predictable $S$, the structure condition until $\infty$ is equivalent to (NA1).
\end{lem}

\begin{proof}
   This is essentially Proposition 3.2 of \cite{Larsen2007} in infinite time. We provide a slightly simplified version of their proof, because later we will need some results obtained during the proof.
   
   Let $Z$ be a supermartingale density. Since $Z$ is strictly positive, it is of the form $dZ_t = Z_{t-}(dL_t - dB_t)$ for a local martingale $L$ and a predictable increasing process $B$. Since $M$ is continuous, there exists a predictable process $\lambda \in L^2_{\text{loc}}(M)$, such that $dL_t = \lambda_t dM_t + dN_t$, where $N$ is a local martingale that is strongly orthogonal to all components of $M$, see \cite{Jacod2003}, Theorem III.4.11. In particular $[H\cdot M, N]$ is a local martingale for every integrand $H$. Moreover
   \begin{align*}
      0 < Z_\infty = Z_0 \mathcal{E}( \lambda \cdot M + N - B)_\infty = Z_0 \mathcal{E}( \lambda \cdot M)_\infty \mathcal{E}(N - B)_\infty,
   \end{align*}
   which is only possible if $\lambda$ satisfies \eqref{eq:structure condition until infty} and if $\mathcal{E}(N-B)_\infty > 0$. It only remains to show that $\lambda$ also satisfies \eqref{eq:structure condition}.
   
   Let $H$ be a 1-admissible strategy. Write $W^H_t = 1 + (H \cdot S)_t$ for the wealth process generated by $H$. Then $W^H Z$ is a nonnegative supermartingale. Since $Z$ is strictly positive, we must have $W^H_t \equiv 0$ for $t \ge \tau^H = \inf\{ s \ge 0: W^H_{s-} = 0 \text{ or } W^H_s = 0\}$. Therefore we may replace $H$ by $H 1_{\{t < \tau^H\}}$ without loss of generality. Define $\pi_t = H_t/W^H_{t-}$, where we interpret $0/0=0$ as before. Then finally
   \begin{align*}
      W^H_t = 1 + (H \cdot S)_t = 1 + \int_0^t \pi_s W^H_{s-} dS_s.
   \end{align*}
   So if we slightly abuse notation and define $W^\pi_t = W^H_t$, then $W^\pi_t = 1 + \int_0^t \pi_s W^\pi_{s-} dS_s$, and every wealth process is of this form.
   
   We write $dX_t \sim dY_t$ if $d(X - Y)_t$ is the differential of a local martingale. Integration by parts applied to $Z W^\pi$ gives
   \begin{align}\label{eq:ibp supermartingale density} \nonumber
      d(Z W^\pi)_t & = W^\pi_{t-} dZ_t + Z_{t-} \pi_t W^\pi_{t-} dS_t + d[(\pi W^\pi_{-} dS), Z]_t \\ \nonumber
      & = W^\pi_{t-} Z_{t-} (\lambda_t dM_t + dN_t - dB_t) + Z_{t-} \pi_t W^\pi_{t-} (dM_t + dA_t)  \\ \nonumber
      & \qquad + W^\pi_{t-} Z_{t-} d[\pi\cdot (M + A), \lambda \cdot M + N - B]_t \\
      & \sim - W^\pi_{t-} Z_{t-} dB_t + Z_{t-} \pi_t W^\pi_{t-} dA_t + W^\pi_{t-} Z_{t-} (d[ \pi \cdot M, \lambda \cdot M]_t - d[\pi \cdot A, B]_t).
   \end{align}
   Of course $[\pi \cdot A, B] \equiv 0$, because $A$ is continuous. But thanks to Proposition I.4.49 c) of \cite{Jacod2003}, \eqref{eq:ibp supermartingale density} has the advantage that it is also valid if $M$ and $A$ are not necessarily continuous (although in that case $\lambda \cdot M$ should be replaced by a general local martingale $N$). 
   
   Let now $C$ and $\sigma$ be as described in Definition \ref{def:structure condition}. Theorem III.4.5 of \cite{Jacod2003} implies
   \begin{align}\label{eq:second ibp supermartingale density} \nonumber
      d(Z W^\pi)_t &\sim W^\pi_{t-} Z_{t-} \left( - dB_t + \pi_t dA_t + d\langle \pi \cdot M, \lambda \cdot M\rangle_t \right) \\
      & =  W^\pi_{t-} Z_{t-}\left( -dB_t + \sum_{i = 1}^d \pi^i_t \left( dA^i_t + \sum_{j=1}^d \sigma^{ij}_t \lambda^j_t dC_t \right)\right).
   \end{align}
   Assume that there exists an $i \in \{1,\dots,d\}$ for which the almost surely continuous process
  \begin{align*}
      D^i_t = \int_0^t \left( Z_{s-} dA^i_s + \sum_{j=1}^d \sigma^{ij}_s \beta^j_s dB_s \right)
   \end{align*}
   is not evanescent.  We claim that then there exists a 1-admissible strategy $\pi$ for which the finite variation part of $(Z W^\pi)$ is increasing on a small time interval. This is a contradiction to $Z W^\pi$ being a supermartingale. By the predictable Radon-Nikodym theorem of Delbaen and Schachermayer (\cite{Delbaen1995a}, Theorem 2.1 b)), there exists a predictable $\gamma^i$ with values in $\{-1, 1 \}$, such that $\int_0^\cdot \gamma^i_s dD^i_s = V^i$, where $V^i$ denotes the total variation process of $D^i$. Note that \cite{Delbaen1995a} work with complete filtrations, but given the $(\F^P_t)$-predictable $\gamma^i$ that they construct, we can apply Lemma \ref{lem:complete predictable optional} to obtain a $(\F_t)$-predictable $\tilde{\gamma}^i$ that is indistinguishable from $\gamma^i$.
   
   Let now  $n\in \N$ and take $\pi_t = n \gamma^i_t$ in \eqref{eq:second ibp supermartingale density}. Then $d(Z W^{\pi})_t \sim W^{\pi}_{t-} Z_{t-}\left( -dB_t + n dV^i_t \right)$, and $V^i$ is an increasing process. Since $\pi$ is bounded, $W^{\pi}_{t-} > 0$ for all $t \ge 0$, and of course also $Z_{t-} > 0$ for all $t \ge 0$. But if $H$ is strictly positive, and $F$ is a finite variation process, then $\int_0^\cdot H_s dF_s$ is a decreasing process if and only if $F$ is a decreasing process. And clearly $- B + nV^i$ can only be decreasing for all $n \in \N$ if $V^i \equiv 0$, a contradiction.

   Therefore $D^i$ is evanescent, and thus for some predictable $\alpha^i$,
   \begin{align*}
      0 \equiv \left( dA^i_t + \sum_{j=1}^d \sigma^{ij}_t \lambda^j_t dC_t \right) = \left( \alpha^i_t d\langle M^i_t \rangle + \sum_{j=1}^d \sigma^{ij}_t \lambda^j_t dC_t \right) =  \left( \alpha^i_t \sigma^{ii}_t + (\sigma_t \lambda_t)^i\right) dC_t
   \end{align*}
   so that
   \begin{align}\label{eq:predictable proof1}
      \alpha^i \sigma^{ii} = -(\sigma \lambda)^i \quad dC(\omega)\otimes P(d\omega)-\text{almost everywhere,}
   \end{align}
   i.e. \eqref{eq:structure condition} is satisfied, and the proof is complete.
   
   The converse direction is easy and follows directly from \eqref{eq:ibp supermartingale density}.
\end{proof}

Next we will show that $S^{T-}$ is a local martingale under the measure $Q$ that is associated to $Z$. But first we observe that if $Z$ is a supermartingale density, then $S Z$ is \emph{not} necessarily a local martingale.

\begin{cor}
   Let $Z$ and $S$ be as in Lemma \ref{lem: structure condition for predictable S}. Then $ZS^i$ is a local supermartingale if and only if $S^i \ge 0$ on the support of the measure $dC$. If $S^i \ge 0$ identically, then $ZS^i$ is a supermartingale.

   $ZS^i$ is a local martingale if and only if $S^i = 0$ on the support of the measure $dC$.
\end{cor}

\begin{proof}
   \eqref{eq:structure condition} and \eqref{eq: explicit supermartingale density} imply that
   \begin{align*}
      d(ZS^i)_t & = Z_{t-}dS^i_t + S^i_{t-} dZ_t + d[S^i,Z]_t  \sim - Z_{t-} S^i_{t-} dB_t.
   \end{align*}
   Nonnegative local supermartingales are supermartingales by Fatou's lemma, and therefore the proof is complete.
\end{proof}

Another immediate consequence of Lemma \ref{lem: structure condition for predictable S} is that in the predictable case, the maximal elements among the supermartingale densities are always local martingales. This is important in the duality approach to utility maximization. For details we refer the reader to \cite{Larsen2007}.

We are now ready to prove Theorem \ref{thm:predictable supermartingale density}.

\begin{cor*}[Theorem \ref{thm:predictable supermartingale density}]
   Let $S$ be a predictable semimartingale, and let $Z$ be a supermartingale density for $S$. Let $T$ be a stopping time and $Q$ be a probability measure, such that  $(Z/E_P(Z_0),T)$ is the Kunita-Yoeurp decomposition of $Q$ wrt $P$. Then $S^{T-}$ is a $Q$-local martingale.
   
   Conversely if $Q \gg P$ with Kunita-Yoeurp decomposition $(Z,T)$ wrt $P$, and if $S^{T-}$ is a local martingale under $Q$, then $Z$ is a supermartingale density for $S$.
\end{cor*}

\begin{proof}
   We first show that $S^{T-}$ is $Q$-a.s. locally bounded. Let $\tilde{\rho}_n = \inf \{ t \ge 0: |S^{T-}_t| \ge n\}$. Since $S^{T-}$ was only required to be right-continuous $P$-a.s. and not identically, $\tilde{\rho}_n$ is not necessarily a stopping time. It is however a $(\F^Q_t)$-stopping time. By Lemma \ref{lem:complete stopping time}, we can find a stopping time $\rho_n$ such that $Q(\rho_n = \tilde{\rho}_n)=1$. Then $\sup_n \rho_n$ is a stopping time, and \eqref{eq: kunita decomposition stopping time} implies that
   \begin{align}\label{eq:predictable supermartingale density pr1}
      Q\left(\sup_n \rho_n < T\right) = E_P\left( Z_{\sup_n \rho_n} 1_{\{\sup_n \rho_n < \infty\}} \right) = 0,
   \end{align}
   because $P(\sup_n \rho_n < \infty) = 0$. But $S^{T-}_t$ is constant for $t \ge T$, and therefore $\{\sup_n \rho_n \ge T\}$ is $Q$-a.s. contained in $\{\sup_n \rho_n = \infty\}$, showing that $S^{T-}$ is $Q$-a.s. locally bounded.
   
   Let $(\sigma_n)$ be a localizing sequence for $M$, where $Z = Z_0 + M - A$. We define $\tau_n = \rho_n \wedge \sigma_n$. Let $H$ be a strategy that is 1-admissible under $Q$. Then we can apply Corollary \ref{cor:yoeurps formula 2} (which of course extends from bounded $Y$ to nonnegative $Y$), to obtain
   \begin{align*}
      E_Q(1+(H\cdot S^{T-})_{\tau_n}) = E_P\left( (1+(H\cdot S)_{\tau_n}) Z_{\tau_n} + \int_0^{\tau_n} (1+(H\cdot S)_{s-})dAs\right).
   \end{align*}
   But now \eqref{eq:second ibp supermartingale density} and \eqref{eq:structure condition} imply that $(1+(H\cdot S)) Z + \int (1+(H\cdot S)_{s-})dAs$ is a nonnegative $P$-local martingale starting in 1, and therefore $E_Q(1+(H\cdot S^{T-})_{\tau_n}) \le 1$. Since $(S^{T-})^{\tau_n}$ is bounded, it must be a martingale.
   
   The only remaining problem is that we only know $Q(\sup_n \tau_n \ge T) = 1$ and not $Q(\sup_n \tau_n = \infty) = 1$. But in fact the same arguments also show that $(S^{T-})^{\rho_n\wedge \tau_m}$ is a martingale for all $n,m \in \N$. Therefore we can apply bounded convergence to obtain for all $s,t \ge 0$
   \begin{align*}
      E_Q( (S^{T-})^{\rho_n}_{t+s} | \F_t) = \lim_{m \rightarrow \infty} E_Q( (S^{T-})^{\rho_n \wedge \tau_m}_{t+s} | \F_t) = \lim_{m \rightarrow \infty} (S^{T-})^{\rho_n \wedge \tau_m}_{t} = (S^{T-})^{\rho_n}_{t}
   \end{align*}
   As we argued above, $Q(\sup_n \rho_n = \infty) = 1$, and therefore $S^{T-}$ is a $Q$-local martingale.
   
   Conversely, let $S^{T-}$ be a $Q$-local martingale, and let $H$ be a 1-admissible strategy for $S$ under $P$. Define $\tau = \inf\{t \ge 0: (H \cdot S^{T-})_t < -1 \}$. Then $P(\tau < \infty) = 0$ and therefore $Q(\tau < T) = 0$ by the same argument as in \eqref{eq:predictable supermartingale density pr1}. In particular $H$ is 1-admissible for $S^{T-}$ under $Q$. Now we can repeat the arguments in \eqref{eq:supermartingale density derivation}, to obtain that ${Z}_t = 1_{\{ t < T\}}/\gamma_t$ is a supermartingale density for $S$, where we denoted ${\gamma}_t = (dP/dQ)|_{\F_t}$.
\end{proof}

\begin{rmk}
   Note that we only used the predictability of $S$ once: it was only needed to obtain
   \begin{align*}
      E_P\left( (1+(H\cdot S)_{\tau_n}) Z_{\tau_n} + \int_0^{\tau_n} (1+(H\cdot S)_{s-})dAs\right) \le 1,
   \end{align*}
   for which we applied (results from the proof of) Lemma \ref{lem: structure condition for predictable S}.
\end{rmk}

\begin{cor}[``Predictable weak fundamental theorem of asset pricing'']\label{thm:predictable ftap}
   Let $(\F_t)$ be the right-continuous modification of a standard system. Let $S$ be a predictable stochastic process that is a.s. right-continuous. Then $S$ satisfies (NA1$_s$) if and and only there exists an enlarged probability space $(\overline{\Omega}, \overline{\F}, (\overline{\F}_t), \overline{P})$ and a dominating measure $\overline{Q} \gg \overline{P}$ with Kunita-Yoeurp decomposition $(\overline{Z}, \overline{T})$ with respect to $\overline{P}$, such that $\overline{S}^{\overline{T}-}$ is a $\overline{Q}$-local martingale.
\end{cor}

\begin{proof}
   It remains to be shown that the existence of $\overline{Q}$ implies that $S$ satisfies (NA1$_s$). But if $\overline{Q}$ exists, then Theorem \ref{thm:predictable supermartingale density} and Theorem \ref{thm: ex supermartingale density} imply that $\overline{S}$ satisfies (NA1) on $(\overline{\Omega}, \overline{\F},$ $(\overline{\F}_t), \overline{P})$. Since this is an enlargement of $(\Omega, \F, (\F_t), P)$, the process $S$ must also satisfies (NA1).
\end{proof}

\begin{rmk}
   We argued above that a predictable process satisfying (NA1) must be continuous. Therefore Corollary \ref{thm:predictable ftap} is not much more general than Ruf \cite{Ruf2010}, where it is shown that a diffusion $S$ that satisfies (NA1) admits a dominating measure $Q$ under which $S^{T-}$ is a local martingale. However one difference is that \cite{Ruf2010} only shows that supermartingale densities that are local martingales correspond to dominating local martingale measures. Here we show that in the predictable case this is in fact true for all supermartingale densities. Also, we show equivalence between (NA1) and the existence of a dominating local martingale measure, and not only that (NA1) implies the existence of $Q$. Of course, as is usually the case for this type of result, the reverse direction is much easier.
\end{rmk}

\subsection{The general case}

We start the treatment of the non-predictable case with two examples that illustrate why it is natural to consider dominating local martingale measures for $S^{T-}$ rather than for $S$.

\begin{ex}
   If $S$ is optional and if $Q$ is a dominating local martingale measure for $S$ rather than for $S^{T-}$, then $S$ does not need to satisfy (NA1): Let $T$ be exponentially distributed with parameter 1 under $Q$. Define $S_t = e^t 1_{\{t < T\}}$ for $t \in [0,1]$. Since time is finite, $S$ is a uniformly integrable martingale. Therefore $dP = S_1 dQ$ is absolutely continuous with respect to $Q$. But under $P$ we have $S_t = e^t$ for all $t \in [0,1]$. So clearly $S$ does not satisfy (NA1) under $P$, although $Q$ is a dominating martingale measure for $S$. Of course $S^{T-}$ is \emph{not} a local martingale under $Q$, because $S^{T-}_t = e^t$.
\end{ex}

Recall that a stopping time $\tau$ is called \emph{foretellable} under a probability measure $P$ if there exists an increasing sequence $(\tau_n)$ of stopping times, such that $P(\tau_n < \tau)=1$ for every $n$, and such that $P(\sup_n \tau_n = \tau)=1$. In this case $(\tau_n)$ is called an \emph{announcing sequence} for $\tau$. Every predictable time is foretellable under any probability measure, see Theorem I.2.15 and Remark I.2.16 of \cite{Jacod2003}.

\begin{ex}
   Let $S$ be a semimartingale under $P$ and let $Q \gg P$ be a dominating measure with Kunita-Yoeurp decomposition $(Z,T)$ wrt $P$. Assume that $T$ is not foretellable under $Q$. Then there exists an adapted process $\tilde{S}$ which is $P$-indistinguishable from $S$, such that $\tilde{S}$ is not a $Q$-local martingale: Let $x \in \R^d$ and define $\tilde{S}^x_t = S_t 1_{\{t < T\}} + x 1_{\{t \ge T\}}$, which is $P$-indistinguishable from $S$ since $P(T=\infty)=1$. If $\tilde{S}^x$ is a $Q$-local martingale, then we can take the localizing sequence $\tau^x_n = \inf\{t \ge 0: |\tilde{S}^x_t| \ge n\}$. 
   Since $T$ is not foretellable under $Q$ and since by the same argument as in \eqref{eq: kunita decomposition stopping time} we have $Q(\lim_{n \rightarrow \infty} \tau^x_n \ge T)=1$, there must exist $n \in \N$ for which $Q(\tau^x_n = T) > 0$. Moreover $\tau^x_n = \tau^y_n$ for all $|x|<n$, $|y| < n$, and therefore
   \begin{align*}
       E_Q(S_0) = E_Q(\tilde{S}^x_{\tau^x_n}) = E_Q(S_{\tau^x_n} 1_{\{\tau^x_n < T\}}) + x Q(\tau^x_n \ge T).
   \end{align*}
   We obtain a contraction by letting $x$ vary through the ball of radius $n-1$.
\end{ex}

These two examples show that given $Q \gg P$, it is important to choose a good version of $S$ if we want to obtain a $Q$-local martingale. All the results obtained so far indicate that this good version should be $S^{T-}$. Maybe somewhat surprisingly, this is not true in general, as we demonstrate in the following example.

\begin{ex}
   Let $(L_t)_{t \in [0,1]}$ be a L\'evy process under $Q$, with jump measure $\nu = \delta_1 + \delta_{-1}$ and drift $b \in \R$. To wit, $L_t = N^1_t - N^2_t + b t$, where $N^1$ and $N^2$ are independent Poisson processes. Let $a > |b|$ and let $\tau$ be an exponential random variable with parameter $a$, such that $\tau$ is independent from $L$. Define $T = \tau$ if $\tau \le 1$, and $T = \infty$ otherwise. Then $(e^{a t} 1_{\{t < T\}})_{t \in [0,1]}$ is a uniformly integrable martingale, and therefore it defines a probability measure $dP = e^a 1_{\{1 < T\}}dQ$. Since $T$ and $L$ are independent, $L$ has the same distribution under $P$ as under $Q$. The Kunita-Yoeurp decomposition of $Q$ wrt $P$ is given by $((e^{-a t})_{t \in [0,1]}, T)$.
   
   We claim that $Z = e^{-a \cdot}$ is a supermartingale density for $L$. Let $(\pi_t W^\pi_{t-})$ be a strategy for $L$, where $W^\pi$ is the wealth process obtained by investing in this strategy. This strategy is $1$-admissible if and only if $|\pi_t| \le 1$ for all $t \in [0,1]$. Moreover we get from \eqref{eq:ibp supermartingale density} that
   \begin{align*}
      d(Z W^\pi)_t \sim - W^\pi_{t-} Z_{t-} a dt + Z_{t-} \pi_t W^\pi_{t-} b dt = W^\pi_{t-} Z_{t-} (\pi_t b - a) dt.
   \end{align*}
   Since $W^\pi Z \ge 0$ and since $\pi_t b - a < 0$ (recall that $a > |b|$), the drift rate is negative. Therefore $Z W^\pi$ is a local supermartingale, and since it is a nonnegative process, it is a supermartingale.
   
   Now $T$ is independent from $L$ under $Q$, and $L$ has no fixed jump times. Hence
   \begin{align*}
      Q(\Delta L_T \neq 0, T < \infty) = \int_{[0,\infty)} Q(\Delta L_t \neq 0) (Q\circ T^{-1})(dt) = 0,
   \end{align*}
   which implies that $L^{T-} = L^T$, and this is clearly no $Q$-local martingale.
\end{ex}

\begin{rmk}
   In the preceding example it is possible to show that the modified process 
   \begin{align}\label{eq:Levy counterexample}
      \tilde{L}_t = L^{T-}_t - \frac{b}{a} 1_{\{t \ge T\}}
   \end{align}
   is a $Q$-martingale.
   
   More generally we expect that given a semimartingale $S$, a supermartingale density $Z$ for $S$, and a measure $Q \gg P$ with Kunita-Yoeurp decomposition $(Z,T)$ wrt $P$, there should always exist a version $\tilde{S}$ that is $P$-indistinguishable from $S$, such that $\tilde{S}$ is a $Q$-local martingale. But as \eqref{eq:Levy counterexample} shows, we will need to take different $\tilde{S}$ for different supermartingale densities. Therefore this approach seems somewhat unnatural, and we will not pursue it further. We rather note that there exists a subclass of supermartingale densities that turn $S^{T-}$ into a local martingale.
\end{rmk}


Note that all three examples had one thing in common: $T$ was not foretellable under $Q$. It turns out that if we assume $T$ to be foretellable, then things get much simpler. But it is well known, and easy to see, that $T$ is foretellable under $Q$ if and only if $Z$ is a $P$-local martingale, see \cite{Follmer1972}, Proposition~(2.1) or \cite{Yoeurp1985}, Theorem 6.


   

Therefore we should look for supermartingale densities that are local martingales. We call these supermartingale densities \emph{local martingale densities}. In the case of a one-dimensional $(S_t)_{t \in [0,T_\infty]}$ with finite terminal time $T_\infty$, it is shown by Kardaras \cite{Kardaras2009}, Theorem 1.1, that local martingale densities exist if and only if (NA1) is satisfied. The proof is in the spirit of the article \cite{Karatzas2007}. Takaoka \cite{Takaoka2012} solves the $d$-dimensional case with finite terminal time. More precisely it is easily deduced from Remark 7 of \cite{Takaoka2012} that for locally bounded $d$-dimensional semimartingale $(S_t)_{t \in [0, T_\infty]}$, (NA1) is satisfied if and only if there exists a local martingale density. Takoaka's proof is based on the insight of Delbaen and Schachermayer \cite{Delbaen1995b}, that a change of num\'{e}raire can induce the (NA) property, even if previously there were arbitrage opportunities in the market. \cite{Takaoka2012} continues to show that a clever choice of num\'{e}raire preserves the (NA1) property, so that then the condition (NA) + (NA1) = (NFLVR) is satisfied, which permits to apply the fundamental theorem of asset pricing \cite{Delbaen1994}.

Of course both \cite{Kardaras2009} and \cite{Takaoka2012} work with complete filtrations, but given a local martingale density $\tilde{Z}$ that is $(\F^P_t)$-adapted, there exists an indistinguishable process $Z$ that is $(\F_t)$-adapted, see Lemma \ref{lem:complete predictable optional}.

%

\begin{lem}
   Let $(S_t)_{t \in [0,T_\infty]}$ be a locally bounded semimartingale on a finite time horizon $T_\infty < \infty$, and let $Z$ be a local martingale density for $S$. Let $T$ be a stopping time and $Q$ be a probability measure, such that  $(Z/E_P(Z_0),T)$ is the Kunita-Yoeurp decomposition of $Q$ wrt $P$. Then $S^{T-}$ is a $Q$-local martingale.
   
   Conversely if $Q \gg P$ with Kunita-Yoeurp decomposition $(Z,T)$ wrt $P$, and if $S^{T-}$ is a local martingale under $Q$, then $Z$ is a supermartingale density for $S$.
\end{lem}

\begin{proof}
   The proof is very similar to the one of Theorem \ref{thm:predictable supermartingale density}. In that proof we only used the predictability of $S$ once, to obtain $E_Q( (H \cdot S)_{\sigma_n}) \le 0$ for all strategies $H$ that are 1-admissible under $Q$. Here $(\sigma_n)$ was a localizing sequence for $M$ under $P$, where $Z = Z_0 + M - A$.
   
   So let $(\sigma_n)$ be a localizing sequence for the local martingale $Z$ under $P$, and let $H$ be a strategy that is $1$-admissible for $S$ under $Q$ (and then also under $P$). We can apply Lemma \ref{lem:yoeurps formula} with $A=0$, to obtain
   \begin{align*}
      E_Q(1+(H\cdot S)_{\sigma_n}) = E_P((1 + (H\cdot S)_{\sigma_n}) Z_{\sigma_n}) \le 1,
   \end{align*}
   because $Z$ is a supermartingale density. Now we can just copy the proof of Theorem \ref{thm:predictable supermartingale density}. 
\end{proof}

We obtained our main result, a weak fundamental theorem of asset pricing.

\begin{cor*}[Theorem \ref{thm:main result}]
   Let $(\F_t)$ be the right-continuous modification of a standard system. Let $S$ be an a.s. locally bounded stochastic process that is a.s. right-continuous. Then $S$ satisfies (NA1$_s$) if and and only there exists an enlarged probability space $(\overline{\Omega}, \overline{\F}, (\overline{\F}_t), \overline{P})$ and a dominating measure $\overline{Q} \gg \overline{P}$ with Kunita-Yoeurp decomposition $(\overline{Z}, \overline{T})$ with respect to $\overline{P}$, such that $\overline{S}^{\overline{T}-}$ is a $\overline{Q}$-local martingale.
\end{cor*}

\begin{rmk}
   There is another subclass of supermartingale densities of which one might expect that they correspond to local martingale measures for $S^{T-}$: the maximal elements among the supermartingale densities. A supermartingale density $Z$ is called maximal if it is indistinguishable from any supermartingale density $Y$ that satisfies $Y_t \ge Z_t$ for all $t \ge 0$. If $S$ is not continuous, then some maximal supermartingale densities are supermartingales and not local martingales, see Example 5.1' of \cite{Kramkov1999}. It turns out that such $Z$ will usually \emph{not} correspond to local martingale measures for $S$. Assume for example that we are in the situation described in Theorem 2.2 of \cite{Kramkov1999}, i.e. we have a dual optimizer $Z$ and a primal optimizer $H$ for a certain utility maximization problem. Then point iii) of this theorem states that $(1+ (H\cdot S))Z$ is a uniformly integrable martingale. If we assume now that $Z$ is not a local martingale, as is the case in Example 5.1' of \cite{Kramkov1999}, and if $(\tau_n)$ is a localizing sequence for the martingale part $M$ of $Z = Z_0 + M - A$, then we obtain from Corollary \ref{cor:yoeurps formula 2}
   \begin{align}\label{eq:rmk after main result}\nonumber
      E_Q((1+(H\cdot S^{T-})_{\tau_n}) Z_{\tau_n}) &= E_P((1+(H\cdot S)_{\tau_n})Z_{\tau_n}) + E_P\left( \int_0^{\tau_n} (1+(H\cdot S)_{s-}) dA_s\right) \\
      & = 1 + E_P\left( \int_0^{\tau_n} (1+(H\cdot S)_{s-}) dA_s\right),
   \end{align}
   where we used that $(1 + (H\cdot S))Z$ is a uniformly integrable martingale. Now, since $H$ is optimal, the wealth process $(1+(H\cdot S)_{s-})$ will be strictly positive with positive probability. Since also $dA \neq 0$ with positive probability, the expectation in \eqref{eq:rmk after main result} is strictly positive for large $n$, and therefore $H \cdot S^{T-}$ cannot be a $Q$-supermartingale, i.e. $S^{T-}$ cannot be a $Q$-local martingale.
\end{rmk}

\section{Relation to filtration enlargements}
\label{sec: jacod criterion}

Here we show that Jacod's criterion for initial filtration enlargements is in fact a criterion for the existence of a universal supermartingale density (to be defined below). We also treat general filtration enlargements. We show that if there exists a universal supermartingale density in an enlarged filtration, then a generalized version of Jacod's criterion is satisfied.

\subsection{Jacod's criterion and universal supermartingale densities}

Let $(\Omega, \F, (\F_t), P)$ be a filtered probability space, and let $(\G^0_t)$ be an initial filtration enlargement of $(\F_t)$, meaning that there exists a random variable $L$ such that $\G^0_t = \F_t \vee \sigma(L)$ for every $t \ge 0$. We define the right-continuous regularization of $(\G^0_t)$ by setting $\G_t = \cap_{s > t} \G^0_s$ for all $t \ge 0$.

Recall that Hypoth\`{e}se $(H')$ is satisfied if all $(\F_t)$-semimartingales are $(\G_t)$-semimartingales.

We now give the classical formulation of Jacod's criterion, see \cite{Jacod1985}. For this purpose we need to assume that $L$ takes its values in a Lusin space, which we denote by $(\mathbb{L}, \mathcal{L})$, where $\mathcal{L}$ is the Borel $\sigma$-algebra on $\mathbb{L}$. In particular the regular conditional distributions
\begin{align*}
   P_t(\omega, d\ell) = P(L \in d\ell | \F_t) (\omega)
\end{align*}
exist for all $t \ge 0$. We write $P_L$ for the distribution of $L$. Jacod's criterion states that Hypoth\`{e}se $(H')$ is satisfied as long as almost surely
\begin{align}\label{eq:jacod}
   P_t(\omega, dx) \ll P_L(dx)
\end{align}
 for every $t \ge 0$.

Below we give an alternative proof of this result and we relate it to the existence of a \emph{universal supermartingale density}.

First observe that Hypoth\`{e}se $(H')$ is satisfied if and only if all nonnegative $(\F_t)$-martingales are $(\G_t)$-semimartingales: This follows by decomposing every $(\F_t)$-local martingale into a sum of a locally bounded local martingale and a local martingale of finite variation, and by observing that every bounded process can be made nonnegative by adding a deterministic constant.



\begin{defn}
   Let $(\G_t)$ be a filtration enlargement of $(\F_t)$. Let $Z$ be an adapted process that is a.s. c\`adl\`ag, such that $P(Z_t > 0)=1$ for all $t \ge 0$. Then $Z$ is called \emph{universal supermartingale density} for $(\G_t)$ if $ZM$ is a $(\G_t)$-supermartingale for every nonnegative $(\F_t)$-supermartingale $M$.
\end{defn}

Note that here we do not require $Z_\infty$ to be positive. This is because local semimartingales are semimartingales, and therefore it suffices to verify the $(\G_t)$-semimartingale property of $M$ on $[0,t]$ for every $t \ge 0$. Hence it suffices if $Z_t > 0$ for every $t \ge 0$.

Also note that $ZM$ must be a $(\G_t)$-supermartingale for every nonnegative $(\F_t)$-supermartingale $M$, and not just for nonnegative $(\F_t)$-martingales. 
This has the advantage that now we see immediately that in finite time every process satisfying (NA1) under $(\F_t)$ satisfies also (NA1) under $(\G_t)$: If $Y$ is a $(\F_t)$-supermartingale density for $S$, then $ZY$ is a $(\G_t)$-supermartingale density for $S$.

The first result of this section shows that Jacod's criterion is not so much a criterion for Hypoth\`ese $(H')$ to hold, but rather a criterion for the existence of a universal supermartingale density.
\begin{prop}\label{prop:jacod}
   Let $(\G_t)$ be an initial enlargement of $(\F_t)$ with a random variable $L$ taking its values in a Lusin space. Assume Jacod's criterion \eqref{eq:jacod} is satisfied. Then there exists a universal supermartingale density for $(\G_t)$.
%
%
\end{prop}
\begin{proof}
   \begin{enumerate}
      \item Define for every $t \ge 0$
         \begin{align*}
            Y_t(\omega, \ell) = \frac{dP_t(\omega, \cdot)}{dP_L}(\ell).
         \end{align*}
         By Doob's disintegration theorem there exists a version of $Y_t$ that is $\F_t \otimes \mathcal{L}$-measurable, see also the proof of Theorem VI.2.10 in \cite{Protter2004}. Let $t, s \ge 0$. We first show that $P \otimes P_L$-almost surely
         \begin{align}\label{eq:Yt zero then Yt+s zero}
            \{(\omega, \ell): Y_t(\omega,\ell) = 0\} \subseteq \{(\omega, \ell): Y_{t+s}(\omega,\ell) = 0\}.
         \end{align}
         Note that $Y_{t+s} \ge 0$, and therefore by Fubini's Theorem and the tower property
         \begin{align*}
            \int_{\Omega \times \mathbb{L}} &1_{\{Y_t(\omega, \ell) = 0\}} Y_{t+s}(\omega, \ell) P\otimes P_L(d\omega, d\ell)  = \int_\Omega \int_\mathbb{L} 1_{\{Y_t(\omega, \ell) = 0\}} P_{t+s}(\omega, d\ell) P(d\omega) \\
            & =  \int_\Omega \int_\mathbb{L} 1_{\{Y_t(\omega, L(\omega)) = 0\}} P(d\omega) =  \int_\Omega \int_\mathbb{L} 1_{\{Y_t(\omega, \ell) = 0\}} P_t(\omega, d\ell) P(d\omega) = 0,
         \end{align*}
         since $P_t(\omega, \cdot)$-almost surely $Y_t(\omega, \cdot) > 0$.
         
      \item Define $\tilde{Z}_t(\omega,\ell) = 1_{\{Y_t(\omega, \ell) > 0\}}/Y_t(\omega, \ell)$ and
         \begin{align*}
            Z_t (\omega) = \tilde{Z}_t(\omega, L(\omega)).
         \end{align*}
         This $Z$ is $(\G_t)$-adapted. Let $M$ be a nonnegative $(\F_t)$-supermartingale. 
%
%
%
%
         Let $t, s \ge 0$, let $A \in \F_t$, and $B \in \mathcal{L}$. Then we can apply the tower property to obtain
         \begin{align*}
            E&\left( 1_{A}1_B(L)M_{t+s}Z_{t+s}\right) = \int_\Omega 1_{A}(\omega) \int_\mathbb{L} 1_B(\ell)M_{t+s}\tilde{Z}_{t+s}(\omega, \ell) P_{t+s}(\omega, d\ell) P(d\omega) \\
            & = \int_\Omega 1_{A}(\omega) \int_\mathbb{L} 1_B(\ell)M_{t+s}(\omega) \frac{Y_{t+s}(\omega, x)}{Y_{t+s}(\omega, x)}1_{\{Y_{t+s}(\omega, x) > 0\}}  P_L(d\ell) P(d\omega) \\
            & \le  \int_\Omega 1_{A}(\omega) \int_\mathbb{L} 1_B(\ell) M_{t+s}(\omega) 1_{\{Y_{t}(\omega, x) > 0\}}  P_L(d\ell) P(d\omega).
         \end{align*}
         In the last step we used \eqref{eq:Yt zero then Yt+s zero}  and that $1_{A}(\omega) 1_B(\ell)M_{t+s}(\omega)$ is $P_L \otimes P$-a.s. nonnegative. Using the $(\F_t)$-supermartingale property of $M$ in conjunction with Fubini's theorem, we obtain
         \begin{align*}
             \int_\Omega &1_{A}(\omega) \int_\mathbb{L} 1_B(\ell) M_{t+s}(\omega) 1_{\{Y_{t}(\omega, \ell) > 0\}}  P_L(d\ell) P(d\omega) \\
             & \le \int_\mathbb{L} 1_B(\ell) \int_\Omega 1_{A}(\omega) M_{t}(\omega) 1_{\{Y_{t}(\omega, \ell) > 0\}}   P(d\omega)P_L(d\ell) \\
             & = \int_\Omega 1_{A}(\omega) \int_\mathbb{L} 1_B(\ell) M_{t}(\omega) \frac{Y_t(\omega, \ell)}{Y_t(\omega, \ell)}1_{\{Y_{t}(\omega, \ell) > 0\}}  P_L(d\ell) P(d\omega) \\
             & = \int_\Omega 1_{A}(\omega) \int_\mathbb{L} 1_B(\ell) M_{t}(\omega) \tilde{Z}_t(\omega, \ell) P_t(\omega, d\ell) P(d\omega) \\
             & = E\left[ 1_{A} 1_B(L) (1 + (H \cdot S)_t) Z_t\right]
         \end{align*}
         The monotone class theorem allows to pass from sets of the form $A \cap L^{-1}(B)$ to general sets in $(\G^0_t)$, and therefore $MZ$ is a $(\G^0_t)$-supermartingale. Taking $M \equiv 1$, we see that also $Z$ is a $(\G^0_t)$-supermartingale.

      \item Let us show that $Z_t$ is $P$-a.s. strictly positive for every $t \ge 0$. For this purpose it suffices to show that $P(\omega: Y_t(\omega, L(\omega))=0) = 0$. By the tower property
         \begin{align*}
            E(1_{\{Y_t(\cdot, L(\cdot)) = 0\}}) & = \int_\Omega \int_\mathbb{L} 1_{\{Y_t(\omega, \ell) = 0\}} P_t(\omega, d\ell) P(d\omega) = 0.
         \end{align*}
      \item  $Z$ is not necessarily right-continuous, and also we did not show yet that $ZM$ is a $(\G_t)$-supermartingale and not just a $(\G^0_t)$-supermartingale. But the construction of a right-continuous universal supermartingale density is now done exactly as in the proof of Theorem \ref{thm:abstract ex supermartingale density}. The supermartingale property under $(\G_t)$ follows from Corollary 2.2.10 of \cite{Ethier1986}, which states that $E(X|\G_t) = \lim_{s \downarrow t} E(X|\G^0_s)$ for every $L^1$-random variable $X$.
    \end{enumerate}
\end{proof}


\begin{rmk}
   If we are only interested whether Hypoth\`ese $(H')$ holds and not whether there exists a universal supermartingale density, then we can also work with the unregularized filtration $(\G^0_t)$. Since Hypoth\`ese $(H')$ holds for $(\G_t)$ and since $(G^0_t)$ is a filtration shrinkage of $(\G_t)$, Stricker's theorem implies that Hypoth\`ese $(H')$ is also satisfied for $(\G^0_t)$.
\end{rmk}

\begin{rmk}\label{rmk:filtration discussion}
   We may replace assumption \eqref{eq:jacod} by $P_t(\omega, d\ell) \gg P_L(d\ell)$ or $P_t(\omega, d\ell) \sim P_L(d\ell)$. In the first case we could use the same proof as for Proposition \ref{prop:jacod} to obtain the existence of a nonnegative martingale $Z$, not necessarily strictly positive, such that $Z M$ is a $(\G_t)$-supermartingale for every nonnegative $(\F_t)$-supermartingale $M$. In particular there exists an absolutely continuous measure $Q \ll P$, such that every locally bounded $(P,(\F_t))$-local martingale is a $(Q,(\G_t))$-local martingale. Since (NA) is related to the existence of absolutely continuous local martingale measures, see \cite{Delbaen1995a}, this indicates that the (NA) property may be stable under initial filtration enlargements that satisfy this ``reverse Jacod condition''. Of course it is much harder to satisfy this assumption, for example it will never be satisfied if $L$ is $\F_t$-measurable for some $t \ge 0$.
   
   In case $P_t(\omega, d\ell) \sim P_L(d\ell)$, the same proof as for Proposition \ref{prop:jacod} yields the existence of an equivalent measure $Q \sim P$, such that every nonnegative $(P,(\F_t))$-supermartingale is a nonnegative $(Q,(\G_t))$-supermartingale. In particular, every locally bounded $(P,(\F_t))$-local martingale is a $(Q,(\G_t))$-local martingale. This condition has been intensely studied by Amendinger, Imkeller and Schweizer \cite{Amendinger1998} as well as Amendinger \cite{Amendinger2000}. Obviously it is harder to satisfy than Jacod's condition or the reverse Jacod condition. In financial applications one may however assume that the knowledge of the ``insider'' is perturbed by a small Gaussian noise that is independent of $\F_\infty$ (or more generally by an independent noise with strictly positive density wrt Lebesgue measure). Then $P_t(\omega, d\ell) \sim P_L(d\ell)$ is always satisfied. If the density of the noise is not strictly positive, then we only obtain $P_t(\omega, d\ell) \ll P_L(d\ell)$.
\end{rmk}

\subsection{Universal supermartingale densities and the generalized Jacod criterion}
Let $(\Omega, \F, (\F_t), P)$ be a filtered probability space and let $\G_t \supseteq \F_t$ be a filtration enlargement. Assume that $\G_t$ is countably generated for every $t \ge 0$, to wit, $\G_t = \sigma(B^t_1, B^t_2, \dots)$. In particular the regular conditional probabilities
\begin{align*}
   P_t(\omega, \cdot)|_{\G_t} = P(\cdot |\F_t)|_{\G_t}(\omega)
\end{align*}
exist. We say that the \emph{generalized Jacod condition} is satisfied if for all $t, s \ge 0$ a.s.
\begin{align*}
   P_{t+s}|_{\G_t}(\omega, \cdot) \ll P_t|_{\G_t}(\omega, \cdot).
\end{align*}
It is known that neither Jacod's condition nor the generalized Jacod condition are necessary for Hypoth\`{e}se $(H')$ to hold. But if we assume that there exists a universal supermartingale density for $(\G_t)$, then the generalized Jacod condition necessarily holds.
\begin{prop}
   Assume that there exists a universal supermartingale density for $(\G_t)$. Then the generalized Jacod condition is satisfied.
\end{prop}
\begin{proof}
   \begin{enumerate}
         \item For every $A \in \F$ the process $M^A_t:=E_P(1_A|\F_t)$, $t \ge 0$, is a nonnegative $(\F_t)$-martingale. Therefore $M^AZ$ is a $(\G_t)$-supermartingale. Fix $t, s \ge 0$. Let $A \in \F_{t+s}$ and $B \in \G_t$. Then for every $n \in \N$
            \begin{align*}
               E\left(1_A 1_B \frac{Z_{t+s}}{Z_t} 1_{\{Z_t \ge 1/n\}}\right) & = E\left( \frac{1_B 1_{\{Z_t \ge 1/n\}}}{Z_t} M^A_{t+s} Z_{t+s} \right) \\
               & \le E\left( \frac{1_B 1_{\{Z_t \ge 1/n\}}}{Z_t} M_{t}^A Z_{t}\right) \\
               & = E(1_A E(1_B 1_{\{Z_t \ge 1/n\}}|\F_t)).
            \end{align*}
            Applying monotone convergence on both sides, we obtain that
            \begin{align*}
               E\left(1_A 1_B \frac{Z_{t+s}}{Z_t}\right) \le E(1_A E_P(1_B |\F_t)).
            \end{align*}
            The same inequality holds if we replace $Z_{t+s}/Z_t$ by a version $\tilde{Z}_{t+s}/\tilde{Z}_t$ that is strictly positive for \emph{every} $\omega \in \Omega$. Since the inequality holds for all $A \in \F_{t+s}$, this implies
            \begin{align*}
               \int 1_B(\omega') \frac{\tilde{Z}_{t+s}}{\tilde{Z}_t}(\omega') P_{t+s}(\omega, d\omega') \le P_t(\omega, B) \text{ for almost every } \omega \in \Omega.
            \end{align*}
            This looks promising. The only problem is that the null set outside of which this inequality holds may depend on $B$.
            
	\item Now we use the assumption that $\G_t$ is countably generated. This means that we can find an increasing sequence of finite partitions
	   \begin{align*}
	      \mathcal{P}^{n} = \cup_{k=1}^{K_n} B_k^n
	   \end{align*}
	   of $\Omega$ such that $\G_t = \sigma(\mathcal{P}^n: n \ge 0)$. 
            Since $\cup_{n \ge 0} \sigma(\mathcal{P}^n)$ is countable, we can choose a null set $N$ such that for all $\omega \in \Omega \backslash N$ and all $B \in \cup_n \sigma(\mathcal{P}^n)$ we have
             \begin{align}\label{eq:generalized jacod proof}
               \int 1_B(\omega') \frac{\tilde{Z}_{t+s}}{\tilde{Z}_t}(\omega') P_{t+s}(\omega, d\omega') \le P_t(\omega, B).
            \end{align}
            Now $B \in \cup_n \sigma(\mathcal{P}^n)$ is stable under finite intersections (it even is an algebra). The monotone class theorem then implies that \eqref{eq:generalized jacod proof} holds for all $B \in \sigma(\mathcal{P}^n: n \ge 0) = \G_t$. Since $\tilde{Z}_{t+s}(\omega')/\tilde{Z}_t(\omega') > 0$ for every $\omega' \in \Omega$, the proof is complete.
   \end{enumerate}
\end{proof}

\begin{cor}\label{cor:generalized jacod complete}
   Let $(\F_t)$ be a filtration for which a continuous local martingale $M$ has the predictable representation property. Assume that under $(\G_t)$, $M$ is of the form
   \begin{align*}
      M_t = \tilde{M}_t + \int_0^t \alpha_s d\langle \tilde{M}\rangle_s
   \end{align*}
   for a $(\G_t)$-local martingale $\tilde{M}$ and a predictable integrand $\alpha \in L^2_{\text{loc}}(\tilde{M})$. Then the generalized Jacod condition holds.
\end{cor}

\begin{proof}
   In this case the stochastic exponential
   \begin{align*}
      Z_t = \exp\left( -\int_0^t \alpha_s d \tilde{M}_s - \frac{1}{2} \int_0^t \alpha^2_s d\langle \tilde{M} \rangle_s \right)
   \end{align*}
   is a universal supermartingale density.
\end{proof}

   Corollary \ref{cor:generalized jacod complete} was previously shown by Imkeller, Pontier and Weisz \cite{Imkeller2001} for the case of initial enlargements and under the stronger assumption
   \begin{align*}
      E\left(\int_0^\infty \alpha^2_s d\langle \tilde{M} \rangle_s\right) < \infty.
   \end{align*}
   
   For simplicity we gave the one-dimensional formulation of Corollary \ref{cor:generalized jacod complete}. Of course the same argument works in the multidimensional setting: if $M = (M^1, \dots, M^d)$ has the predictable representation under $(\F_t)$, and if $M$ satisfies the structure condition under $(\G_t)$, then the generalized Jacod condition is satisfied.

\begin{appendix}

\section{Convex compactness and Tychonoff's theorem}

In Section \ref{sec: supermartingale densities} we needed Tychonoff's theorem for products of convexly compact spaces. Since we only applied the result to a countable product of metric spaces, we will only prove this special case. Let us recall the following definitions.

\begin{defn}
   \begin{enumerate}
    \item A set $A$ is called \emph{directed} if it is partially ordered and if for every $a, b \in A$ there exists $c \in A$ such that $a \le c$ and $b \le c$.
    \item Let $X$ be a topological space. A \emph{net} in $X$ is a map from some directed set $A$ to $X$.
    \item A net $\{x_\alpha\}_{\alpha \in A}$ in $X$ \emph{converges} to a point $x \in X$ if for every open neighborhood $U$ of $x$ there exists $\alpha \in A$, such that $x_{\alpha'} \in U$ for every $\alpha' \ge \alpha$.
   \end{enumerate}
\end{defn}

\begin{ex}
   If $A = \N$, then a net in $X$ is just a sequence with values in $X$.
\end{ex}

Zitkovic \cite{Zitkovic2010} introduces the notation $\mathrm{Fin}(A)$, which denotes all non-empty finite subsets of a given set $A$. If $B$ is a subset of a vector space, then $\mathrm{conv}(B)$ denotes the convex hull of $B$. Zitkovic then gives the following definition.

\begin{defn}\label{def:subnet convex}
   Let $\{x_\alpha\}_{\alpha \in A}$ be a net in a topological vector space $X$. A net $\{y_\beta\}_{\beta \in B}$ is called a \emph{subnet of convex combinations of $\{x_\alpha\}_{\alpha \in A}$} if there exists a map $D: B \rightarrow \mathrm{Fin}(A)$ such that
   \begin{enumerate}
    \item $y_\beta \in \mathrm{conv}\{x_\alpha: \alpha \in D(\beta)\}$ for every $\beta \in B$, and
    \item for every $\alpha \in A$ there exists $\beta \in B$ such that $\alpha' \ge \alpha$ for all $\alpha' \in \cup_{\beta' \ge \beta} D(\beta')$.
   \end{enumerate}
\end{defn}

\begin{lem}\label{rmk:subnet of subnet}
   Let $\{y_\beta\}_{\beta \in B}$ be a subnet of convex combinations of $\{x_\alpha\}_{\alpha \in A}$, and let $\{z_\gamma\}_{\gamma \in C}$ be a subnet of convex combinations of $\{y_\beta\}_{\beta \in B}$. Then $\{z_\gamma\}_{\gamma \in C}$ is a subnet of convex combinations of $\{x_\alpha\}_{\alpha \in A}$.
\end{lem}

\begin{proof}
   Let $D_B: B \rightarrow \mathrm{Fin}(A)$ and $D_C: C \rightarrow \mathrm{Fin}(B)$ be two maps as described in Definition \ref{def:subnet convex}, $D_B$ for $\{y_\beta\}_{\beta \in B}$ and $D_C$ for $\{z_\gamma\}_{\gamma \in C}$. Define
   \begin{align*}
      D: C \rightarrow \mathrm{Fin}(A), \qquad D(\gamma) = \cup_{\beta \in D_C(\gamma)} D_B(\beta).
   \end{align*}
   Then we have for all $\gamma \in C$
   \begin{align*}
      z_\gamma \in \mathrm{conv}\{y_\beta: \beta \in D_C(\gamma)\} \subseteq \mathrm{conv}\{x_\alpha: \alpha \in \cup_{\beta \in D_C(\gamma)} D_B(\beta)\} = \mathrm{conv} \{x_\alpha: \alpha \in D(\gamma)\},
   \end{align*}
   and therefore condition (1) of Definition \ref{def:subnet convex} is satisfied. As for condition (2), let $\alpha \in A$. Then there exists $\beta \in B$, such that $\alpha'\ge \alpha$ for all $\alpha' \in \cup_{\beta'\ge \beta} D_B(\beta')$. For this given $\beta$, there exists $\gamma \in C$, such that $\beta' \ge \beta$ for all $\beta' \in \cup_{\gamma'\ge \gamma} D_C(\gamma')$. Hence $\alpha' \ge \alpha$ for all
   \begin{align*}
      \alpha' \in \cup_{\gamma'\ge \gamma} D(\gamma') = \cup_{\gamma'\ge \gamma} \cup_{\beta' \in D_C(\gamma')} D_B(\beta') \subseteq \cup_{\beta' \ge \beta} D_B(\beta').
   \end{align*}
\end{proof}

One of Zitkovic's \cite{Zitkovic2010} main results is the following Proposition.

\begin{prop}\label{prop:characterization convex compactness}
   A closed and convex subset $C$ of a topological vector space $X$ is convexly compact if and only if for any net $\{x_\alpha: \alpha \in A\}$ in $C$ there exists a subnet of convex combinations $\{y_\beta: \beta \in B\}$, such that $\{y_\beta\}$ converges to some $y \in X$.
\end{prop}

We will use this insight to prove a weak version of Tychonoff's theorem for convexly compact sets.

\begin{prop}\label{prop: tychonoff result}
   Let $\{X_n: n \in \N\}$ be a \emph{countable} family of convexly compact \emph{metric} spaces. Then $\prod_{n \in \N} X_n$ is convexly compact in the product topology.
\end{prop}

\begin{proof}
   Let $\{(x_\alpha(n))_{n\in\N}: \alpha \in A\}$ be a net in $\prod_{n \in \N} X_n$. Then $\{x_\alpha(1): \alpha \in A\}$ is a net in $X_1$. By Proposition \ref{prop:characterization convex compactness}, there exists a subnet of convex combinations $\{(y^1_\beta(n))_{n \in \N}: \beta \in B_1\}$, such that $\{y^1_\beta(1): \beta \in B_1\}$ converges to some $y(1) \in X_1$. We can now inductively construct for every $k > 1$ a subnet of convex combinations
   \begin{align*}
      \{(y^k_\beta(n))_{n \in \N}: \beta \in B_k\} \text{ of } \{(y^{k-1}_\beta(n))_{n \in \N}: \beta \in B_{k-1}\},
   \end{align*}
   such that $\{y^k_\beta(k): \beta \in B_k\}$ converges to some $y(k) \in X_k$. By Remark \ref{rmk:subnet of subnet}, every $\{(y^k_\beta(n))_{n \in \N}: \beta \in B_k\}$ is a subnet of convex combinations of $\{x_\alpha: \alpha \in A\}$. We denote the corresponding maps from $B_k$ to $\mathrm{Fin}(A)$ by $D_k$. Note that by construction, $\{y^k_\beta(l): \beta \in B_k\}$ converges to $y(l)$ for all $1 \le l \le k$. Now take the directed set $\N \times A$ with the partial order $(k, \alpha) \le (k', \alpha')$ if $k \le k'$ and $\alpha \le \alpha'$. We write $d_l$ for the distance on $X_l$. Define for $(k,\alpha) \in \N\times A$ the set of ``admissible indices'' as
   \begin{align*}
      B(k, \alpha) = \left\{\beta \in B_k:\alpha' \ge \alpha \text{ for all } \alpha' \in D_k(\beta) \text{ and } d_l(y^k_\beta(l), y(l)) \le \frac{1}{k}\text{ for } l=1,\dots, k \right\}.
   \end{align*}
   By construction of the $(y^k_\beta(n))_{n \in \N}$, every $B(k,\alpha)$ is non-empty. For every $(k, \alpha) \in \N \times A$ choose $\beta(k, \alpha)\in B(k,\alpha)$. Note that here we explicitly apply the Axiom of Choice! Set $z_{(k,\alpha)} = y^k_{\beta(k, \alpha)}$ and $D((k,\alpha)) = D_k(\beta(k,\alpha))$. Then by construction, $\{z_{(k, \alpha)}: (k, \alpha) \in \N \times A\}$ is a subnet of convex combinations of $\{x_\alpha: \alpha \in A\}$, which converges to some $(y(n))_{n \in \N} \in \prod_{n \in \N} X_n$ in the product topology. Therefore $\prod_n X_n$ is convexly compact in the product topology.
\end{proof}

\begin{rmk}
   The proof is surprisingly technical considering that we have a countable product of metric spaces. In this case compactness is equivalent to sequential compactness, and therefore the proof of Tychonoff's theorem is a triviality, based on a diagonal sequence argument. But so far there seems to be no characterization of convex compactness in terms of sequential compactness, and therefore we had to work with nets rather than with sequences.
\end{rmk}

\section{The question of complete filtrations}\label{app:complete filtration}

Here we collect some classical observations that allow to transfer results of other authors that were obtained under complete filtrations to our setting. We follow \cite{Jacod2003}.

Let $(\Omega, \F, (\F_t)_{t \ge 0}, P)$ be a filtered probability space with a right-continuous filtration $(\F_t)$. Write $\F^P$ for the $P$-completion of $\F$, and $\mathcal{N}^P$ for the $P$-null sets of $\F^P$. Then $\F_t^P = \F_t \vee \mathcal{N}^P$ is the $\sigma$-algebra generated by $\F_t$ and $\mathcal{N}^P$, and $(\F^P_t)$ satisfies the usual conditions. We call $(\Omega, \F^P, (\F^P_t), P)$ the \emph{completion} of $(\Omega, \F, (\F_t), P)$.

Recall that the optional $\sigma$-algebra over $(\F_t)$ is the $\sigma$-algebra on $\Omega \times [0,\infty)$ that is generated by all processes of the form $X_t(\omega) = 1_A(\omega) 1_{[r,s)}(t)$ for some $0\le r < s < \infty$ and $A \in \F_r$. The predictable $\sigma$-algebra over $(\F_t)$ is the $\sigma$-algebra on $\Omega \times [0,\infty)$ that is generated by all processes of the form $X_t(\omega) = 1_A(\omega) 1_{\{0\}}(t) + 1_B(\omega) 1_{(r,s]}(t)$ for some $0\le r < s < \infty$, for $A \in \F_0$, and $B \in \F_r$. Similarly we define the predictable and optional $\sigma$-algebra over $(\F^P_t)$. 

The first result relates stopping times under $(\F_t)$ and under $(\F^P_t)$.

\begin{lem}[Lemma I.1.19 of \cite{Jacod2003}]\label{lem:complete stopping time}
   Any stopping time on the completion $(\Omega, (\F^P_t))$ is a.s. equal to a stopping time on $(\Omega, (\F_t))$.
\end{lem}

Next, we show a similar result on the level of processes.

\begin{lem}\label{lem:complete predictable optional}
   Any predictable (resp. optional) process on the completion $(\Omega, (\F^P_t))$ is indistinguishable from a predictable (resp. optional) process on $(\Omega, (\F_t))$.
\end{lem}

\begin{proof}
   The predictable case is Lemma I.2.17 of \cite{Jacod2003}. The proof of the optional case works exactly in the same way: the claim is trivial for the generating processes described above, and we can use the monotone class theorem to pass to indicator functions of general optional sets. Then we use monotone convergence to pass to general optional processes.
\end{proof}

This allows us to obtain a similar result for c\`adl\`ag processes.

\begin{lem}\label{lem:complete cadlag}
   Let $S$ be a $(\F^P_t)$-adapted process that it a.s. c\`adl\`ag. Then $S$ is indistinguishable from a $(\F_t)$-adapted process (which is then of course 
   a.s. c\`adl\`ag as well).
\end{lem}

\begin{proof}
   Since $(\F^P_t)$ is complete, $S$ admits an indistinguishable version $\tilde{S}$ that is $(\F_t^P)$-adapted and c\`adl\`ag for \emph{every} $\omega \in \Omega$. This $\tilde{S}$ is optional, so now the result follows from Lemma \ref{lem:complete predictable optional}.
\end{proof}

\end{appendix}

\textbf{Acknowledgement:} N.P. thanks Asgar Jamneshan for the introduction to  filtration enlargements. We are grateful to Stefan Ankirchner, Kostas Kardaras, and Johannes Ruf for their helpful comments on an earlier version of this paper. We thank Alexander Gushchin for pointing out the reference \cite{Rokhlin2010}. Part of the research was carried out while the authors were visiting the Department of Aerospace Engineering at the University of Illinois at Urbana-Champaign. We are grateful for the hospitality at UIUC. N.P. is supported by a Ph.D. scholarship of the Berlin Mathematical School.

\bibliography{bib}

\providecommand{\bysame}{\leavevmode\hbox to3em{\hrulefill}\thinspace}
\providecommand{\MR}{\relax\ifhmode\unskip\space\fi MR }
\providecommand{\MRhref}[2]{%
  \href{http://www.ams.org/mathscinet-getitem?mr=#1}{#2}
}
\providecommand{\href}[2]{#2}
\begin{thebibliography}{\v{Z}02}

\bibitem[ADI06]{Ankirchner2006}
Stefan Ankirchner, Steffen Dereich, and Peter Imkeller, \emph{{The Shannon
  information of filtrations and the additional logarithmic utility of
  insiders}}, Ann. Probab. \textbf{34} (2006), no.~2, 743--778.

\bibitem[ADI07]{Ankirchner2007}
\bysame, \emph{{Enlargement of filtrations and continuous Girsanov-type
  embeddings}}, {S{\'e}minaire de Probabilit{\'e}s XL}, Springer, 2007,
  pp.~389--410.

\bibitem[AIS98]{Amendinger1998}
J{\"u}rgen Amendinger, Peter Imkeller, and Martin Schweizer, \emph{{Additional
  logarithmic utility of an insider}}, Stochastic Process. Appl. \textbf{75}
  (1998), no.~2, 263--286.

\bibitem[Ame00]{Amendinger2000}
J{\"u}rgen Amendinger, \emph{{Martingale representation theorems for initially
  enlarged filtrations}}, Stochastic Process. Appl. \textbf{89} (2000), no.~1,
  101--116.

\bibitem[Ank05]{Ankirchner2005}
Stefan Ankirchner, \emph{{Information and semimartingales}}, Ph.D. thesis,
  Humboldt-Universit{\"a}t zu Berlin, 2005.

\bibitem[Bec01]{Becherer2001}
Dirk Becherer, \emph{{The numeraire portfolio for unbounded semimartingales}},
  Finance Stoch. \textbf{5} (2001), no.~3, 327--341. \MR{1849424 (2002k:91068)}

\bibitem[CFR12]{Carr2011}
Peter Carr, Travis Fisher, and Johannes Ruf, \emph{{On the hedging of options
  on exploding exchange rates}}, {arXiv:1202.6188} (2012).

\bibitem[DM80]{Dellacherie1980}
Claude Dellacherie and Paul-Andre Meyer, \emph{{Probabilit{\'e}s et potentiel.
  Chapitres V {\`a} VIII: Th{\'e}orie des martingales}}, Actualites
  scientifiques et industrielles, 1385. Publications de l'Institut de
  Mathematique de l'Universite de Strasbourg, XVII. Paris: Hermann. XVIII,
  1980.

\bibitem[DS94]{Delbaen1994}
Freddy Delbaen and Walter Schachermayer, \emph{{A general version of the
  fundamental theorem of asset pricing}}, Math. Ann. \textbf{300} (1994),
  no.~1, 463--520.

\bibitem[DS95a]{Delbaen1995}
\bysame, \emph{{Arbitrage possibilities in Bessel processes and their relations
  to local martingales}}, Probab. Theory Related Fields \textbf{102} (1995),
  no.~3, 357--366.

\bibitem[DS95b]{Delbaen1995a}
\bysame, \emph{{The existence of absolutely continuous local martingale
  measures}}, Ann. Appl. Probab. \textbf{5} (1995), no.~4, 926--945.

\bibitem[DS95c]{Delbaen1995b}
\bysame, \emph{{The no-arbitrage property under a change of num{\'e}raire}},
  Stochastics Stochastics Rep. \textbf{53} (1995), 213--226.

\bibitem[EK86]{Ethier1986}
Stewart~N. Ethier and Thomas~G. Kurtz, \emph{{Markov processes:
  Characterization and convergence}}, John Wiley \& Sons, 1986.

\bibitem[FG06]{Follmer2006}
Hans F{\"o}llmer and Anne Gundel, \emph{{Robust projections in the class of
  martingale measures}}, Illinois J. Math. \textbf{50} (2006), no.~1-4,
  439--472.

\bibitem[FI93]{Follmer1993}
Hans F{\"o}llmer and Peter Imkeller, \emph{{Anticipation cancelled by a
  Girsanov transformation: a paradox on Wiener space}}, Ann. Inst. Henri
  Poincar{\'e} Probab. Stat. \textbf{29} (1993), no.~4, 569--586.

\bibitem[F{\"o}l72]{Follmer1972}
Hans F{\"o}llmer, \emph{{The exit measure of a supermartingale}}, Probab.
  Theory Related Fields \textbf{21} (1972), no.~2, 154--166.

\bibitem[HP81]{Harrison1981}
J.Michael Harrison and Stanley~R. Pliska, \emph{{Martingales and stochastic
  integrals in the theory of continuous trading.}}, Stochastic Processes Appl.
  \textbf{11} (1981), 215--260 (English).

\bibitem[IPW01]{Imkeller2001}
Peter Imkeller, Monique Pontier, and Ferenc Weisz, \emph{{Free lunch and
  arbitrage possibilities in a financial market model with an insider}},
  Stochastic Process. Appl. \textbf{92} (2001), no.~1, 103--130.

\bibitem[Jac79]{Jacod1979}
Jean Jacod, \emph{{Calcul stochastique et problemes de martingales}}, vol. 714,
  Springer, 1979.

\bibitem[Jac85]{Jacod1985}
\bysame, \emph{{Grossissement initial, hypothese (H') et theoreme de
  Girsanov}}, {Grossissements de filtrations: exemples et applications; Lecture
  Notes in Mathematics Vol. 1118}, Springer, 1985, pp.~15--35.

\bibitem[JPS10]{Jarrow2010}
Robert~A. Jarrow, Philip Protter, and Kazuhiro Shimbo, \emph{{Asset price
  bubbles in incomplete markets}}, Math. Finance \textbf{20} (2010), no.~2,
  145--185.

\bibitem[JS03]{Jacod2003}
Jean Jacod and Albert~N. Shiryaev, \emph{{Limit theorems for stochastic
  processes}}, 2nd ed., Springer, 2003.

\bibitem[Kar10]{Kardaras2010}
Constantinos Kardaras, \emph{{Finitely additive probabilities and the
  fundamental theorem of asset pricing}}, {Contemporary quantitative finance.
  Essays in honour of Eckhard Platen.}, Springer, 2010, pp.~19--34.

\bibitem[Kar12]{Kardaras2009}
\bysame, \emph{{Market viability via absence of arbitrage of the first kind}},
  Finance Stoch. \textbf{16} (2012), no.~4, 1--15.

\bibitem[KK07]{Karatzas2007}
Ioannis Karatzas and Constantinos Kardaras, \emph{{The num{\'e}raire portfolio
  in semimartingale financial models}}, Finance Stoch. \textbf{11} (2007),
  no.~4, 447--493.

\bibitem[KKN11]{Kardaras2011}
Constantinos Kardaras, D{\"o}rte Kreher, and Ashkan Nikeghbali, \emph{{Strict
  local martingales and bubbles}}, arXiv:1108.4177 (2011), 1--34.

\bibitem[KP11]{Kardaras2011a}
Constantinos Kardaras and Eckhard Platen, \emph{{On the semimartingale property
  of discounted asset-price processes.}}, Stochastic Processes Appl.
  \textbf{121} (2011), no.~11, 2678--2691 (English).

\bibitem[KS99]{Kramkov1999}
Dmitry Kramkov and Walter Schachermayer, \emph{{The asymptotic elasticity of
  utility functions and optimal investment in incomplete markets}}, Ann. Appl.
  Probab. \textbf{9} (1999), no.~3, 904--950.

\bibitem[Kun76]{Kunita1976}
Hiroshi Kunita, \emph{{Absolute continuity of Markov processes}},
  {S{\'e}minaire de Probabilit{\'e}s X}, Springer, 1976, pp.~44--77.

\bibitem[LW00]{Loewenstein2000}
Mark Loewenstein and Gregory~A. Willard, \emph{{Local martingales, arbitrage,
  and viability - Free snacks and cheap thrills}}, Econom. Theory \textbf{161}
  (2000), 135--161.

\bibitem[L{\v{Z}}07]{Larsen2007}
Kasper Larsen and Gordan {\v{Z}}itkovi{\'c}, \emph{{Stability of
  utility-maximization in incomplete markets.}}, Stochastic Processes Appl.
  \textbf{117} (2007), no.~11, 1642--1662.

\bibitem[Mey72]{Meyer1972}
Paul~A. Meyer, \emph{{La mesure de H. F{\"o}llmer en th{\'e}orie des
  surmartingales}}, {S{\'e}minaire de Probabilit{\'e}s VI}, Springer, 1972,
  pp.~118--129.

\bibitem[PP10]{Pal2010}
Soumik Pal and Philip Protter, \emph{{Analysis of continuous strict local
  martingales via h-transforms}}, Stochastic Process. Appl. \textbf{120}
  (2010), no.~8, 1424--1443.

\bibitem[Pro04]{Protter2004}
Philip~E. Protter, \emph{{Stochastic integration and differential equations}},
  2nd ed., Springer, 2004.

\bibitem[Rok10]{Rokhlin2010}
Dmitry~B. Rokhlin, \emph{{On the existence of an equivalent supermartingale
  density for a fork-convex family of random processes}}, Math. Notes
  \textbf{87} (2010), no.~3--4, 556--563.

\bibitem[Ruf12]{Ruf2010}
Johannes Ruf, \emph{{Hedging under arbitrage}}, Math. Finance, to appear
  (2012).

\bibitem[RY99]{Revuz1999}
Daniel Revuz and Marc Yor, \emph{{Continuous martingales and Brownian motion}},
  3rd ed., Springer, 1999.

\bibitem[Sch95]{Schweizer1995}
Martin Schweizer, \emph{{On the minimal martingale measure and the
  F{\"o}llmer-Schweizer decomposition}}, Stoch. Anal. Appl. \textbf{13} (1995),
  573--599.

\bibitem[Tak12]{Takaoka2012}
Koichiro Takaoka, \emph{{On the condition of no unbounded profit with bounded
  risk}}, Finance Stoch., to appear (2012).

\bibitem[\v{Z}02]{Zitkovic2002}
Gordan \v{Z}itkovi{\'c}, \emph{{A filtered version of the bipolar theorem of
  Brannath and Schachermayer.}}, J. Theor. Probab. \textbf{15} (2002), no.~1,
  41--61 (English).

\bibitem[Yan80]{Yan1980}
Jia-An Yan, \emph{{Caract{\'e}risation d'une classe d'ensembles convexes de
  ${L^1}$ ou ${H^1}$}}, {S{\'e}minaire de Probabilit{\'e}s XIV}, Springer,
  1980, pp.~220--222.

\bibitem[Yoe85]{Yoeurp1985}
Chantha Yoeurp, \emph{{Th{\'e}or{\`e}me de Girsanov g{\'e}n{\'e}ralis{\'e} et
  grossissement d'une filtration}}, {Grossissements de filtrations: exemples et
  applications; Lecture Notes in Mathematics Vol. 1118}, Springer, 1985,
  pp.~172--196.

\bibitem[{\v{Z}}it10]{Zitkovic2010}
Gordan {\v{Z}}itkovi{\'c}, \emph{{Convex compactness and its applications}},
  Math. Financ. Econ. \textbf{3} (2010), no.~1, 1--12.

\end{thebibliography}
\bibliographystyle{amsalpha}

\end{document}